\documentclass[reqno]{amsart}
\usepackage{amssymb}

\newcommand{\lt}{\left}
\newcommand{\rt}{\right}
\newcommand{\half}{\frac{1}{2}}
\newcommand{\sumkinf}[2]{\sum_{#1 = #2}^\infty}
\newcommand{\re}{\operatorname{Re}}
\newcommand{\im}{\operatorname{Im}}
\newcommand{\res}[1]{\operatorname{Res}_{#1}}
\newcommand{\conj}[1]{\overline{#1}}
\newcommand{\inttoinf}[1]{\int_{#1}^\infty}

\newcommand{\intconst}{\frac{1}{2 \pi i}}
\newcommand{\intline}[1]{\int_{(#1)}}
\newcommand{\invmellin}[1]{\intconst \intline{#1}}
\newcommand{\dblinvmellin}[2]{\left(\intconst \right)^2 \intline{#1} \intline{#2}}
\newcommand{\triinvmellin}[3]{\left(\intconst \right)^3 \intline{#1} \intline{#2} \intline{#3}}
\newcommand{\quadinvmellin}[4]{\left(\intconst \right)^4 \intline{#1} \intline{#2} \intline{#3} \intline{#4}}
\newcommand{\betaF}[3]{\frac{\Gamma(#1) \Gamma(#2)}{\Gamma(#3)}}
\newcommand{\qqquad}{\qquad \qquad}
\newcommand{\qqqquad}{\qqquad \qqquad}
\newtheorem{theorem}{Theorem}
\newtheorem{proposition}{Proposition}
\newtheorem{lemma}{Lemma}
\newtheorem*{remark}{Remark}
\usepackage{color}
\usepackage{enumitem}
\usepackage[usenames,dvipsnames]{xcolor}

\usepackage{hyperref}
\numberwithin{equation}{section}
\numberwithin{equation}{subsection}
\numberwithin{lemma}{section}
\numberwithin{proposition}{section}
\begin{document}
\title{Hybrid Bounds on Twisted L-Functions Associated to Modular Forms}
\author{Chan Ieong Kuan}
\date{\today}
\begin{abstract}
  For $f$ a primitive holomorphic cusp form of even weight $k \geq 4$, level $N$, and $\chi$ a Dirichlet character mod $Q$ with $(Q,N) = 1$, we establish the following hybrid subcovex bound for $t \in \mathbb{R}$:
\[ L(\half + it, f_\chi) \ll Q^{\frac{3}{8} + \varepsilon} (1+|t|)^{\frac{1}{3-2\theta} + \varepsilon} \]
where $\theta$ is the best bound toward the Ramanujan-Petersson conjecture at the infinite place. The implied constant only depends on $f$ and $\varepsilon$. This is done via amplification and taking advantage of a shifted convolution sum of two variables as defined and analyzed in \cite{Jeff}.
\end{abstract}
\maketitle


\section{Introduction} \label{sec:intro}
\subsection{Hybrid Bounds}
The growth of $L$-functions on the critical line $\re s = \half$  has been one of the most studied problems in analytic number theory. This paper is concerned with $L$-functions of a holomorphic cusp form $f$, twisted by a character $\chi$ of conductor $Q$. By using functional equation and Phragm\'en-Lindel\"of principle, one can obtain the convexity bound
\[ L(\half+it, f_\chi) \ll (Q(1+|t|))^{\half+\varepsilon}, \]
where we suppress the level and weight aspects here.

Throughout the years, there have been many attempts at lowering the exponents, most of which have focused on one chosen aspect. Since our result concerns $Q$- and $t$-aspects, we will state some known results in these directions.

In the $t$-aspect, Good showed in \cite{good1983square} that for $f$ a holomorphic cusp form of the full modular group,
\[ L(\half+it,f) \ll (1+|t|)^{\frac{1}{3} + \varepsilon} \]
Meurman showed the same result for $f$ Maass forms of full modular group in \cite{meurman1990}. For number fields, subconvexity results were proved in Petridis and Sarnak \cite{PetridisSarnak2001} and Diaconu and Garrett \cite{Diaconu2010}.

In the $Q$-aspect, the first subconvexity result was obtained by Duke, Friedlander and Iwaniec \cite{Duke1993} for holomorphic cusp forms of full level. Later, Bykovskii showed in \cite{bykovskii1998trace} that for general level,
\[ L(\half+it, f_\chi) \ll_t Q^{\frac{3}{8} + \varepsilon}, \]
with a polynomial dependence in $(1+|t|)$, provided that the nebentypus of $f$ is trivial. This same bound without the nebentypus restriction is obtained in Hoffstein and Hulse \cite{Jeff}, and Blomer and Harcos \cite{blomer2008hybrid}. In \cite{blomer2008hybrid}, $f$ can also be taken as a Maass form.

Hhybrid bounds in $Q$- and $t$-aspects have been worked on by Blomer and Harcos in \cite{blomer2008hybrid}, Munshi \cite{munshi2012circle} and Han Wu \cite{Wu2012Burgess}, which, following the method of Michel and Venkatesh, uses amplification. The bound obtained is:
\[ L(\half+it,f_\chi) \ll (Q(1+|t|))^{\frac{3}{8} + \frac{\theta}{4} + \varepsilon}, \]
where no complementary series with parameter $>\theta$ appears as a component of a cuspidal automorphic representation of $GL_2(\mathbb{A})$.

One thing to note is that these hybrid bounds do not reach the best known exponents in the $t$-aspect. In this work, we partially resolve this situation by proving the following result.
\begin{theorem} \label{thm_main}
For $f$ a primitive holomorphic cusp form of even weight $k \geq 4$, level $N$, and $\chi$ a Dirichlet character mod $Q$, where $(Q,N) = 1$, we have
  \[ L(\half + it, f_\chi) \ll (1+|t|)^{\frac{1}{3-2\theta} + \varepsilon} Q^{\frac{3}{8}+\frac{\theta}{4} + \varepsilon}, \]
where $\theta$ is a bound toward the Ramanujan-Petersson conjecture. 
\end{theorem}
\begin{remark}
  A bound of $\theta$ for congruence subgroups of $\operatorname{SL}_2(\mathbb{Z})$ is $\frac{7}{64}$ by Kim and Sarnak \cite{Kim:2003aa}. It should be noted that our theorem does not currently cover the case of Maass forms as the corresponding shifted convolution is not analyzed yet.
\end{remark}
Our work also uses amplification. The major difference between this work and \cite{Wu2012Burgess} is that we treat non-Archimedian and Archimedian places differently while Wu treated them uniformly. As such, we need more precise control on the $t$-aspect, which is achieved by relating the problem to the shifted convolution sum of two variables analyzed in \cite{Jeff}.

\subsection{Structure of this paper}
Our goal is to bound $L(\half + it, f_\chi)$ in the $Q$ and $t$-aspects. First, we quote relevant results from \cite{Jeff} in section \ref{sec:prep}. We then apply amplification methods in section \ref{sec:amp}, reducing the problem to understanding the growth in $Q$ and $t$ of the following expression, where $G$ and $\mathcal{L}$ are amplification parameters and $\alpha = \frac{1}{\log (Q(1+|t|))}$.
\begin{align*}
&\varphi(Q) G \sum_{\substack{l_1,l_2 \sim \mathcal{L} \\ l_1,l_2 \text{ prime}}} \conj{\chi(l_1)} \chi(l_2) \sum_{\substack{m_1,m_2 \geq 1 \\ m_1l_1 \equiv m_2l_2 (Q)}} \frac{A(m_1)\conj{A(m_2)}(l_1l_2)^{\alpha}}{m_1^{\half+it} m_2^{\half-it}} \notag \allowdisplaybreaks[0]\\
 &\phantom{\varphi(Q) \sum_{l_1,l_2} \conj{\chi(l_1)}\chi(l_2)} \qquad \qquad \qquad \times   V \lt( \frac{m_1}{x} \rt) V \lt( \frac{m_2}{x} \rt) e^{-G \lt| \log \lt(\tfrac{l_1m_1}{l_2m_2} \rt) \rt|} 
\end{align*}
We then separate the analysis of this expression into the ``diagonal" portion ($m_1l_1 = m_2l_2$) and two ``off-diagonal" portions ($m_1l_1 = m_2l_2 + h_0Q$ and $m_2l_2 = m_1l_1 + h_1Q$, for $h_0, h_1 \geq 1$).

In section \ref{sec:diag}, we analyze the diagonal term with inverse Mellin transforms (propositions \ref{eq_S_d1_prop} and \ref{eq_S_d2_propo}).

For the off-diagonal portions, the first thing to note is they have the same contribution up to conjugation. Our analysis relies heavily on the shifted convolution sum of two variables $Z_Q(s,w)$ from \cite{Jeff}. By inverse Mellin transforms, we relate the off-diagonal term to a four-fold integral involving $Z_Q(s,w)$. This is done in section \ref{sec:off-d-setup} (proposition \ref{eq_S_o1_int}).

The analysis of the off-diagonal then splits into the discrete part and the continuous part, due to the fact that such a splitting exists for $Z_Q(s,w)$. The analysis of each part is done by moving lines of integration, with the primary goal of reducing the $x$-exponent as much as possible, and a secondary goal of reducing contribution of the $t$-aspect where possible. The results can be found in propositions \ref{eq_S_o1d_Zres_propo}, \ref{eq_S_o1d_Z_sp2}, and \ref{propo_total_cts}.

In the last section, we put the results of the previous sections together. Choosing $\mathcal{L}$ and $G$ optimally yields the theorem.

\section*{Acknowledgements}
  I would like to thank my Ph.D. advisor Jeffrey Hoffstein at Brown University for suggesting me this problem and making his manuscript available. I would also like to thank Min Lee for many helpful conversations and Gergely Harcos for pointing out a few mistakes in an older version of this paper.

\section{Preparations}\label{sec:prep}
Throughout this paper, fix a holomorphic cusp form $f$ of even weight $k \geq 4$, level $N$:
\[ f(z) = \sumkinf{n}{1} A(n) n^{\frac{k-1}{2}} e(nz) = \sumkinf{n}{1} a(n) e(nz), e(z) = e^{2\pi iz} \]

\subsection{Shifted convolution of two variables}
The most crucial object of this paper is the shifted convolution of two variables $Z_Q(s,w)$, analyzed in \cite{Jeff}. We quote several of the results here for convenience. Fix $\ell_1,\ell_2$ be primes relatively prime to $NQ$ and of size $\mathcal{L}$. The definition of $Z_Q(s,w)$ is as follows:
\begin{equation} \label{eq_Zq_defn2}
   Z_Q(s,w) := \sum_{\substack{h_0,m_2 \geq 1 \\ m_1l_1 = m_2l_2 + h_0Q}} \frac{A(m_1) \conj{A(m_2)} \left( 1+\frac{h_0Q}{l_2m_2} \right)^{\frac{k-1}{2}}}{(l_2m_2)^s (h_0Q)^{w + \frac{k-1}{2}}}
\end{equation}
In \cite{Jeff}, it is shown that
\[ \lim_{\delta \to 0} Z_Q(s,u;\delta) = Z_Q(s,u) \]
and that $Z_Q(s,u;\delta)$ has the following spectral expansion
\begin{align}
  &Z_Q(s,s'-s-\tfrac{k}{2} + 1;\delta) \notag\\ =& \frac{(4\pi)^k (l_1l_2)^{\frac{k-1}{2}} 2^{s-\half}}{\Gamma(s+k-1) 2\sqrt{\pi}} \Bigg( \sum_j L_Q(s',\conj{u_j}) M(s,t_j,\delta) \conj{\langle U, u_j \rangle} \notag \\
  &\qquad + \mathcal{V}_{N[l_1,l_2]} \sum_{\mathfrak{a}} \invmellin{0} \zeta_{\mathfrak{a},Q}(s',-z) M(s,\tfrac{z}{i},\delta) \conj{ \langle U,E_\mathfrak{a}(*,\half+z) \rangle} \,d  z \Bigg) \label{eq_Z_decomp}
\end{align}
where the $\mathfrak{a}$-sum is over cusps of $\Gamma_0(N[l_1,l_2])$ and
\begin{align}
  L_Q(s',\conj{u_j}) :&= \sum_{h \geq 1} \frac{\conj{\rho_j(-hQ)}}{(hQ)^{s'}}, \notag \\
  \zeta_{\mathfrak{a},Q}(s',z) :&= \sum_{h \geq 1} \frac{\rho_\mathfrak{a}(-hQ,z)}{2(hQ)^{s'}}, \notag \\
  U(z) :&= y^k f(l_1z) \conj{f(l_2z)}, \label{eq_L_kappa_U} \\
  \mathcal{V}_{N} :&= \frac{\pi [\operatorname{SL}_2(\mathbb{Z}):\Gamma_0(N)]}{3} \notag
\end{align}
$\rho_j(n)$ being $n$-th Fourier coefficient of Maass form $u_j$ and $\rho_{\mathfrak{a}}(n,z)$ being $n$-th Fourier coefficient of Eisenstein series at cusp $\mathfrak{a}$ with holormophic argument at $\half + z$. $[l_1,l_2]$ denotes the least common multiple of $l_1$ and $l_2$.

In \cite{Jeff}, it is also shown that in \eqref{eq_Z_decomp}, if we are to take the limit as $\delta$ goes to $0$, we actually require $\re s < \half - \frac{k}{2}$ for the sum and integral there to be absolutely convergent. 

\subsection{Some useful analytic information}
The properties of the $M$ and $Z$ functions that are relevant for this work are quoted in the following two propositions:
\begin{proposition} \label{eq_M_s_poles}
Let $z \in \mathbb{C} - \half \mathbb{Z}$. Then $M(s,\tfrac{z}{i};\delta)$ has simple poles at $s = \half \pm z - r$, for $r$ a nonnegative integer. We denote the following:
  \[ \res{s=\half \pm z - r} M(s,\tfrac{z}{i}; \delta) = c_r(\pm z; \delta), \]
where $c_r(\pm z,\delta)$ has the following explicit expression as $\delta \to 0$:
  \[ \lim_{\delta \to 0} c_r(\pm z, \delta) = \frac{(-1)^r \sqrt{\pi} 2^{\mp z + r} \Gamma(\pm 2z - r) \Gamma(\half \mp z + r)}{r! \Gamma(\half + z) \Gamma(\half - z)} \]
And we have the following values at $z = \pm \half$, as $\delta \to 0$:
  \begin{align}
     c_r(-\half;\delta) &\to -\frac{2^{r+\half} \sqrt{\pi}}{2(r+1)!}  \text{ for } r \geq 0 \\
     c_0(\half; \delta) &\to \sqrt{\frac{\pi}{2}} \\
     c_r(\half; \delta) &\to \frac{2^{r-\half} \sqrt{\pi}}{2r!} \text{ for } r \geq 1
  \end{align}
Also for $\re (s+z) \leq \half + \max(0, |\re z|)$, with $s$ and $z$ at least a distance $\varepsilon$ away from poles,
  \begin{equation} \label{eq_M_delta0}
    \lim_{\delta \to 0} M(s,\tfrac{z}{i}; \delta) = \frac{\sqrt{\pi} 2^{\half - s} \Gamma(s-\half+z) \Gamma(s-\half-z) \Gamma(1-s)}{\Gamma(\half+z) \Gamma(\half - z)}
  \end{equation}
\end{proposition}

\begin{proposition} \label{eq_Z_s_poles}
$Z(s,u;\delta)$ has simple poles at $s = \half \pm it_j - r$, where $r \in \mathbb{Z}_{\geq 0}$. Taking the residue at those points and $\delta \to 0$, we have the following:
  \[ \lim_{\delta \to 0} \res{s=\half + it_j - r} Z(s,s'-s-\tfrac{k}{2} + 1;\delta) = (l_1l_2)^{\frac{k-1}{2}} c_{r,j} L_Q(s'; \conj{u_j}), \]
where $c_{r,j}$ has growth, when $\widetilde{T} \gg 1$,
  \begin{equation} \label{eq_crj_growth}
    \sum_{|t_j| \sim \widetilde{T}} |c_{r,j}|^2 e^{\pi|t_j|} \ll \log (\widetilde{T}) (l_1l_2)^{-k} \widetilde{T}^{2r+1}
  \end{equation}
\end{proposition}

We will also give an explicit expression for $\zeta_{\mathfrak{a},Q}(s',-z)$:
\begin{proposition} \label{eq_Zeta_explicit}
  For cusp $\mathfrak{a}=\frac{b}{c}$ of $\Gamma_0(N)$, we write $c = c_0 c_1$, where $(c_0, \frac{N}{c}) = 1$. Then we have the following:
  \begin{align}
    &\zeta_{\mathfrak{a},Q}(s',-z) \notag \\
    =& \frac{\pi^{\frac{1}{2}-z}}{\Gamma(\frac{1}{2}-z)} \left( \frac{(c,\frac{N}{c})}{cN}\right)^{\frac{1}{2}-z} \left( \frac{c}{(c,\frac{N}{c})} \right)^{1-s'-z} \frac{Q^{-(s'+z)}}{\varphi((c,\frac{N}{c}))} \\
    &\qquad \qquad \times \sum_{\chi ((c,\frac{N}{c}))} \frac{\conj{\chi(-Qu)} L^{((c,\frac{N}{c}))}(s'+z,\conj{\chi}) L^{(\frac{N}{c})}(s'-z,\chi)}{L^{(N)}(1-2z,\chi^2)} \sum_{n | c_1^\infty} \frac{G_n(\chi)}{n^{s'+z}} \notag \\
    &\qquad \qquad \times \prod_{p | c_0} (1-p^{-(1-z-s')}\chi(p)) \prod_{p^\gamma \| Q} (\sigma_{2z}^{(\chi^2)}(p^\gamma) - \conj{\chi(p)} p^{-(s'-z)} \sigma_{2z}^{(\chi^2)}(p^{\gamma-1})) \notag
  \end{align}
  where $n | c_1^\infty$ means that $n$ runs over all integers such that $n | c_1^k$ for sufficiently large integer $k$,
  \begin{align*}
    L^{(c)}(s,\chi) &= \prod_{p \nmid c} (1-\chi(p)p^{-s})^{-1} \\
    \sigma_{z}^{\chi}(n) &= \sum_{d | n} \chi(d) d^z \\
    G_n(\chi) &= \sum_{\substack{a \pmod{(c,\frac{N}{c})} \\ (a,c,\frac{N}{c}) = 1}} \chi(a) e \left(\frac{a n}{(c,\frac{N}{c})} \right), e(x) = e^{2\pi ix}
  \end{align*}
  For $\re (s'+z) > 0$ and $\re (s'-z) > 0$, the only poles come from the trivial character term.
\end{proposition}

\subsection{Some miscellaneous estimates and bounds}
We will also require the following estimate concerning $L$-functions:
\begin{proposition}
  For $\re s' \geq \half$ and $\theta$ being a bound toward Ramanujan-Petersson conjecture,
  \begin{equation} \label{eq_L_growth}
    \sum_{|t_j| \sim \widetilde{T}} |L_Q(s', \conj{u_j})|^2 e^{-\pi |t_j|} \ll Q^{-2s' + 2\theta} \mathcal{L}^{2+\varepsilon} (1+|s'|+ \widetilde{T})^{2+\varepsilon}
  \end{equation}
\end{proposition}

Together with \eqref{eq_crj_growth} and the following fact
\[ \sum_{|t_j| \sim \widetilde{T}} |\langle U, u_j \rangle|^2 e^{\pi |t_j|} \ll \mathcal{L}^{-2k} \widetilde{T}^{2k} \log t,\]
we have the following proposition:
\begin{proposition} \label{propo_j_sums}
  With the same notations as before,
  \begin{align}
    \sum_{|t_j| \sim \widetilde{T}} L_Q(s', \conj{u_j}) \conj{\langle U, u_j \rangle} &\ll Q^{-s'+\theta} \mathcal{L}^{1-k+\varepsilon} \widetilde{T}^{1+k+\varepsilon} \label{eq_L_inner} \\
    \sum_{|t_j| \sim \widetilde{T}} L_Q(s', \conj{u_j}) c_{r,j} &\ll Q^{-s'+\theta} \mathcal{L}^{1-k+\varepsilon} \widetilde{T}^{\frac{3}{2}+r+\varepsilon} \label{eq_L_crj}
  \end{align}
\end{proposition}
These equations can be proved by Cauchy's inequality with facts quoted before the proposition. 

We will finally note some equivalent facts with Eisenstein series and a very particular functional equation involving Eisenstein series at the $0$-cusp:
\begin{proposition} \label{inner_prod_props}
  The inner product $\langle U, E_\mathfrak{a}(*,s) \rangle$ has the following properties:
  \begin{enumerate}[leftmargin=1.5em]
    \item $\displaystyle \res{s=1} \langle U, E_\mathfrak{a}(*,s) \rangle = \frac{(l_1l_2)^{-\frac{k-1}{2}}}{l_1 \mathcal{V}_{Nl_1}} \langle f,f \rangle$ if $l_1 = l_2$.
    \item $\displaystyle \res{s=1} \langle U, E_\mathfrak{a}(*,s) \rangle = \frac{(l_1l_2)^{-\frac{k-1}{2}}}{l_1l_2 \mathcal{V}_{Nl_1l_2}} \langle f,f \rangle E_{l_1,l_2}(1)$ if $l_1 \neq l_2$.
    \item $\displaystyle [l_1,l_2] \mathcal{V}_{N[l_1,l_2]} \int_{-T}^T \sum_{\mathfrak{a}} \zeta_{\mathfrak{a},Q}(\frac{1}{2},-it) \conj{\langle U, E_\mathfrak{a}(*,\frac{1}{2}+it) \rangle} \,dt \ll Q^{-s'} \mathcal{L}^{1-k+\varepsilon} T^{1+k+\varepsilon}$
  \end{enumerate}
\end{proposition}

\begin{proposition} \label{FE_Eisen_0}
  The Eisenstein series at $0$-cusp has the following functional equation:
  \[ E_0(z,s) = \sum_{\mathfrak{a} = \frac{b}{c}} \frac{\sqrt{\pi} \Gamma(s-\frac{1}{2})}{\Gamma(s)} \frac{\varphi(\frac{N}{c})}{\varphi((c,\frac{N}{c}))} \left( \frac{(c,\frac{N}{c})}{N \cdot \frac{N}{c}} \right)^s \frac{\zeta^{(c)}(2s-1)}{\zeta^{(N)}(2s)} E_\mathfrak{a}(z,1-s) \]
\end{proposition}

\section{Amplifying both aspects} \label{sec:amp}
Our aim here is to understand the growth of $L(\half +it, f_\chi)$ in $t$- and $Q$-aspects. Since it is sufficient to prove the result on eigenforms, so we assume $f$ is an eigenform. We perform our investigation by averaging around $\half + it$ for a small interval as well as applying the amplification technique.

For this, we choose a rapidly decreasing function $V: \mathbb{R} \to \mathbb{R}$ such that its Mellin transform $v(s)$ is meromorphic between $-5 < \re s < 5$. Moreover, $v(s)$ should only have a simple pole at $s=0$ with residue $1$ and exponential decay in $\im s$ as $\im s \to \infty$. An example of this is $v(s) = \frac{1}{5} \Gamma \lt(\frac{s}{5} \rt)$. Specifying $v(s)$ is enough, as:
\begin{equation}
  V(x) = \invmellin{2} v(s) x^{-s} \,d s \label{eq_easy_mellin}
\end{equation}

We start by writing the $L$-function as a rapidly converging series:
\begin{lemma} \label{lemma_approx}
  As $x \to \infty$,
  \begin{equation} \label{eq_approx}
     L\lt( \half + it, f_\chi \rt) = \sum_{n \geq 1} \frac{A(n)\chi(n)}{n^{\half + it}} V \lt( \frac{n}{x} \rt) + O(x^{-\varepsilon}).
  \end{equation}
\end{lemma}

\begin{proof}
  Consider the following inverse Mellin transform:
  \begin{equation}
    I_0 := \invmellin{2} L \lt( \half + it + s, f_\chi \rt) v(s) x^s \,d s
  \end{equation}
  On the one hand, since the argument of $L$-function in $I_0$ is in the region of absolute convergence, we have:
  \begin{align*}
    I_0= \invmellin{2} \sum_{n \geq 1} \frac{A(n) \chi(n)}{n^{\half + s + it}} v(s) x^s \,d s = \sum_{n \geq 1} \frac{A(n) \chi(n)}{n^{\half + it}}  V \lt( \frac{n}{x} \rt)
  \end{align*}
  On the other hand, we can move the line of integration to $-\half - \varepsilon$, picking up the only simple pole at $s = 0$ and obtain:
  \begin{equation*}
    I_0 = L\lt(\half + it, f_\chi \rt) + O(x^{-\varepsilon})
  \end{equation*}
  Putting the two equivalent expressions of $I_0$ together proves the lemma.
  
\end{proof}

Our aim here is to get the bound on $L(\half + it, f_\chi)$. To this end, we first amplify the character and obtain:
\begin{equation} \label{eq_most_important}
   |L(\half + it, f_\chi)|^2 |\sum_{l \sim \mathcal{L}} 1|^2 \leq \sum_{\psi (Q)} |L(\half+it,f_\psi)|^2 |\sum_{l \sim \mathcal{L}} \conj{\chi(l)}\psi(l)|^2 \end{equation}
where the first summation runs over all Dirichlet characters $\psi$ mod $Q$ and the $l$-sums are running over primes that are relatively prime to $QN$. The parameter $\mathcal{L}$ is to be chosen optimally later, subject to $\mathcal{L} < Q$.

Next we perform amplification on the $t$-aspect, with modified ideas based upon \cite{Hansen}. The result is the following:
\begin{lemma} \label{lemma_one_use}
\begin{align}
  &|L(\half+it,f_\psi)|^2 |\sum_{l \sim \mathcal{L}} \conj{\chi(l)} \psi(l)|^2 \notag \\
 \ll & \log^{4}(Q(1+|t|)) \!\!\!\!\! \int \limits_{|r| \leq A} \!\!\!\!\! \inttoinf{-\infty} \!\!\!\! |L(\half+i(t+r+r'),f_\psi)|^2 |\sum_{l \sim \mathcal{L}} \conj{\chi(l)} \psi(l) l^{\alpha - ir'} |^2 \,\frac{d r' d r}{\pi (1+\lt( \frac{r'}{G} \rt)^2)} \notag \\
  & \qquad \qquad + O(\log^3 (Q(1+|t|))) + O(|\sum_{l \sim \mathcal{L}} \conj{\chi(l)} \psi(l)|^2), \label{eq_amp_t_single}
\end{align}
where $A := \sqrt{10 \log (Q(1+|t|))}$, $\alpha := \frac{1}{\log (Q(1+|t|))}$ and $G \geq 2A$.
\end{lemma}
\begin{remark}
  In the proof, we will see that the introduction of $G$ into the integral is via the positivity of the integrand. This leads to the desire to minimize $G$ subject to the constraint above.
\end{remark}

\begin{proof}
The proof relies on estimating $L(\sigma + it, f_\psi)$ by averaging $L(\sigma - \alpha + ir, f_\psi)$ over $r$ in a small interval centered around $t$. Each integral expression defined below is essentially illustrating this fact.

First we will show that $L(\half + it, f_\psi)$ is approximable by averaging the $L$-function over a small interval. To this end, consider the following integral:
\begin{equation} \label{eq_I1_defn}
  I_1 := \invmellin{2} L(\half + it + s, f_\psi) \frac{e^{s^2}}{s} \,d s
\end{equation}
On the one hand, $I_1$ is $O(1)$ by bounding the $L$-function by a constant. On the other hand, if we move the line of integration down to $\re s = -\alpha$, then we have:
\[ I_1 = L(\half+it, f_\psi) + \invmellin{-\alpha} L(\half+it+s,f_\psi) \frac{e^{s^2}}{s} \,d  s. \]
When put together with \eqref{eq_I1_defn}, we derive that
\[ L(\half+it,f_\psi) = O(1) + \frac{1}{2\pi} \inttoinf{-\infty} L(\half-\alpha+it+ir, f_\psi) \frac{e^{(-\alpha+ir)^2}}{\alpha-ir} \,d  r. \]
After taking absolute values and squaring both sides, one gets:
\begin{equation} \label{eq_critline}
 | L(\half+it,f_\psi) |^2 \ll \lt( \inttoinf{-\infty} | L(\half-\alpha+it+ir, f_\psi)|  \frac{e^{\alpha^2-r^2}}{\sqrt{\alpha^2 + r^2}} \,d  r \rt)^2 + O(1)
\end{equation}
To continue our investigation, we will split the integral into two parts, $|r| \leq A$ and $|r| > A$, where $A = \sqrt{10 \log (Q(1+|t|))}$.

We start by examining the part of the integral with $|r| > A$, applying convexity for the $L$-function:
\begin{align*}
   \int_{|r| > A} |L(\half-\alpha+it+ir, f_\psi)| \frac{e^{\alpha^2 - r^2}}{\sqrt{\alpha^2 + r^2}} \,d  r &\ll \int_{|r| > A} (Q|t+r|)^{\half+\alpha} \frac{e^{\alpha^2 - r^2}}{r} \,d  r \\
  & \ll (Q(1+|t|))^{\half + \alpha} e^{\alpha^2 - A^2} \ll 1
\end{align*}

For the part $|r| \leq A$, we apply Cauchy's inequality and functional equation:
\begin{align*}
  &\lt( \int_{|r| \leq A} |L(\half-\alpha+it+ir, f_\psi)| \frac{e^{\alpha^2 - r^2}}{\sqrt{\alpha^2 + r^2}} \,d  r \rt)^2 \\
 \leq &\int_{|r| \leq A} |L(\half-\alpha+it+ir, f_\psi)|^2 \,d  r \cdot  \int_{|r| \leq A} \frac{e^{2\alpha^2 - 2r^2}}{\alpha^2 + r^2} \,d  r \\
 \ll &\alpha^{-2} (Q(1+|t|))^{4\alpha} \int_{|r| \leq A} |L(\half+\alpha+it+ir, f_\psi)|^2 \,d  r
\end{align*}

Putting these into \eqref{eq_critline}, recalling $A = \sqrt{10 \log (Q(1+|t|))}$ and $\alpha = \frac{1}{\log (Q(1+|t|))}$, we get
\begin{align}
  |L(\half+it,f_\psi)|^2 &\ll  \log^{2}(Q(1+|t|)) \int_{|r| \leq A} |L(\half+\alpha + i(t+r) ,f_\psi)|^2 \,d  r + O(1) \notag
\end{align}
We will multiply both sides by $|\sum_{l \sim \mathcal{L}} \conj{\chi(l)} \psi(l)|^2$, obtaining:
\begin{align}
   &|L(\half+it,f_\psi)|^2 |\sum_{l \sim \mathcal{L}} \conj{\chi(l)} \psi(l)|^2 \notag \\
  \ll  &\log^{2}(Q(1+|t|)) \!\!\!\! \int \limits_{|r| \leq A} \!\!\!\! |L(\half+\alpha + i(t+r) ,f_\psi)|^2 |\sum_{l \sim \mathcal{L}} \conj{\chi(l)} \psi(l)|^2 \,d  r \!+\! O(|\sum_{l \sim \mathcal{L}} \conj{\chi(l)} \psi(l)|^2) \label{eq_critline_1}
\end{align}

Now, we approximate $L(\half + \alpha + it_2, f_\psi) \sum_{l \sim \mathcal{L}} \psi(l) \conj{\chi(l)}$ by an integral over a small interval on the critical line. To achieve this, we construct the following auxilary integral:
\begin{equation} \label{eq_I2_defn}
  I_2 := \invmellin{2} L(\half + \alpha + it_2 + s, f_\psi) \lt(\sum_{l \sim \mathcal{L}} \conj{\chi(l)} \psi(l) l^{-s} \rt) \frac{e^{s^2}}{s} \,d  s
\end{equation}
Doing the same analysis as before and combining expressions, we obtain:
\begin{align}
  &|L(\half+it,f_\psi)|^2 |\sum_{l \sim \mathcal{L}} \conj{\chi(l)} \psi(l)|^2 \notag \\
 \ll & \log^{4}(Q(1+|t|)) \int \limits_{|r| \leq A} \int \limits_{|r'| \leq A_r} |L(\half+i(t+r)+ir',f_\psi)|^2 |\sum_{l \sim \mathcal{L}} \conj{\chi(l)} \psi(l) l^{\alpha - ir'} |^2 \,d  r' \,d  r \notag \\
 &\qquad \qquad + O(\log^2 (Q(1+|t|)) A) + O(|\sum_{l \sim \mathcal{L}} \conj{\chi(l)} \psi(l)|^2) \label{eq_final_form_1}
\end{align}
where $A_r = \sqrt{10 \log(Q|t+r|)}$. Note that all inequalities are independent of $Q$, $L$ and $t$, as long as $L \ll Q$. Also note that we can make $A_r$ uniform by enlarging the region to $|r'| \leq \sqrt{40 \log(Q(1+|t|))} = 2A$, the inequality still holding due to positivity of the integrand.

Continuing to use the positivity of the integrand, for $G \geq 3A$, the integrand is bounded by:
{
\allowdisplaybreaks
\begin{align}
  & \int_{|r| \leq A} \int_{|r'| \leq 2A} |L(\half+i(t+r)+ir',f_\psi)|^2 |\sum_{l \sim \mathcal{L}} \conj{\chi(l)} \psi(l) l^{\alpha - ir'} |^2 \,d  r' \,d  r \notag \\
 \ll & \int_{|r| \leq G} \! |L(\half+i(t+r),f_\psi)|^2 |\sum_{l \sim \mathcal{L}} \conj{\chi(l)} \psi(l) l^{\alpha - ir} |^2 \,\frac{d r}{\pi (1+\lt( \frac{r}{G} \rt)^2)} \notag \\
 \ll & \int_{|r| \leq A} \inttoinf{-\infty} |L(\half+it+ir,f_\psi)|^2 |\sum_{l \sim \mathcal{L}} \conj{\chi(l)} \psi(l) l^{\alpha - ir} |^2 \,\frac{dr}{\pi (1+\lt( \frac{r}{G} \rt)^2)} \label{eq_final_form_2}
\end{align}
}
Putting this into \eqref{eq_final_form_1} gives the proposition. 
\end{proof}

Putting this proposition together with \eqref{eq_most_important}, we derive the following:
\begin{proposition} \label{propo_amp}
With the same values of $A$, $\alpha$ and the same constraint on $G$ as in lemma \ref{lemma_one_use},
\begin{align}
  &|L(\half + it, f_\chi)|^2 |\sum_{l \sim \mathcal{L}} 1|^2 \notag \\
 \ll & \log^{4}(Q(1+|t|)) \inttoinf{-\infty} \sum_{\psi (Q)} |L(\half+i(t+r),f_\psi)|^2 |\sum_{l \sim \mathcal{L}} \conj{\chi(l)} \psi(l) l^{\alpha - ir} |^2 \,\frac{d r}{\pi (1+\lt( \frac{r}{G} \rt)^2)} \notag \\
 &+ O(Q \log^3 (Q(1+|t|))) + O(Q\mathcal{L}) \label{eq_final_form_3}
\end{align}
\end{proposition}

\begin{proof}
  The only part that requires a proof is the last error term. In particular, we should show that
  \[ \sum_{\psi (Q)} |\sum_{l \sim \mathcal{L}} \conj{\chi(l)} \psi(l)|^2 \ll Q\mathcal{L} \]
  Starting with the left-hand side, we have:
  \begin{align*}
    \sum_{\psi(Q)} |\sum_{l \sim \mathcal{L}} \conj{\chi(l)} \psi(l)|^2 &= \sum_{\psi(Q)} \sum_{l_1, l_2 \sim \mathcal{L}} \conj{\chi(l_1)} \psi(l_1) \chi(l_2) \conj{\psi(l_2)} = \varphi(Q) \sum_{l} 1\ll Q\mathcal{L}
  \end{align*}
  The second equality is obtained by summing over the characters, which implies that $l_1 = l_2$ since $l_1 \equiv l_2 (Q)$ and $\mathcal{L} < Q$.

\end{proof}

Our next immediate goal is to execute the character sum and the $r$-integral in \eqref{eq_final_form_3}. Replacing the $L$-series with \eqref{eq_approx}, up to $O(x^{-\varepsilon})$, one obtains:
\begin{equation} \label{eq_S_defn}
  S := \inttoinf{-\infty} \sum_{\psi (Q)} \left| \sum_{m \geq 1} \sum_{l \sim \mathcal{L}} \frac{A(m)\psi(m)\psi(l) \conj{\chi(l)} l^{\alpha} }{m^{\half+it + ir} l^{ir}} V\lt(\frac{m}{x} \rt) \right|^2 \,\frac{d r}{\pi (1+\lt( \frac{r}{G} \rt)^2)}
\end{equation}
We apply Parseval here to obtain:
{\allowdisplaybreaks
\begin{align*}
  S &= \inttoinf{-\infty} \sum_{\substack{a \text{ mod } Q \\ (a,Q) = 1}} \varphi(Q) \left| \sum_{m \geq 1} \sum_{\substack{l \sim \mathcal{L} \\ l \text{ prime} \\ ml \equiv a (Q)}} \frac{A(m) \conj{\chi(l)} l^{\alpha}}{m^{\half+it + ir} l^{ir}} V\lt(\frac{m}{x} \rt) \right|^2 \,\frac{d r}{\pi (1+\lt( \frac{r}{G} \rt)^2)} \\
    &\leq \inttoinf{-\infty} \sum_{a \text{ mod } Q} \varphi(Q) \left| \sum_{m \geq 1} \sum_{\substack{l \sim \mathcal{L} \\ l \text{ prime} \\ ml \equiv a (Q)}} \frac{A(m) \conj{\chi(l)} l^{\alpha}}{m^{\half+it + ir} l^{ir}} V\lt(\frac{m}{x} \rt) \right|^2 \,\frac{d r}{\pi (1+\lt( \frac{r}{G} \rt)^2)} \\
   &= \varphi(Q) \inttoinf{-\infty} \sum_{m_1, m_2 \geq 1} \sum_{\substack{l_1,l_2 \sim \mathcal{L} \\ l_1,l_2 \text{ prime} \\ m_1l_1 \equiv m_2l_2 (Q)}} \frac{A(m_1)\conj{A(m_2)} \conj{\chi(l_1)}\chi(l_2) (l_1l_2)^{\alpha}}{m_1^{\half+it+ir} m_2^{\half-it-ir} l_1^{ir} l_2^{-ir}} \allowdisplaybreaks[0] \\
  &\phantom{=\varphi(Q)} \qqqquad \qqquad \qquad \times  V\lt(\frac{m_1}{x} \rt)V\lt(\frac{m_2}{x} \rt) \,\frac{d r}{\pi (1+\lt( \frac{r}{G} \rt)^2)} \\
  &=\varphi(Q) G \sum_{l_1,l_2} \conj{\chi(l_1)} \chi(l_2) \sum_{\substack{m_1,m_2 \geq 1 \\ m_1l_1 \equiv m_2l_2 (Q)}} \frac{A(m_1)\conj{A(m_2)}(l_1l_2)^{\alpha}}{m_1^{\half+it} m_2^{\half-it}} \notag \allowdisplaybreaks[0]\\
 &\phantom{=\varphi(Q) \sum_{l_1,l_2} \conj{\chi(l_1)}\chi(l_2)} \qqquad \qqquad \times  V\lt(\frac{m_1}{x} \rt) V\lt(\frac{m_2}{x} \rt) e^{-G \lt| \log \lt(\tfrac{l_1m_1}{l_2m_2} \rt) \rt|} \\
  &=: S_d + S_{o_1} + S_{o_2}
\end{align*}
}%
where the diagonal portion $S_d$ and the off-diagonal portions, $S_{o_1}$ and $S_{o_2}$, are defined as follows:
{\allowdisplaybreaks
\begin{align}
  S_d :&= \varphi(Q) G \sum_{l_1,l_2} \conj{\chi(l_1)} \chi(l_2) \sum_{\substack{m_1,m_2 \geq 1 \\ m_1l_1 = m_2l_2}} \frac{A(m_1)\conj{A(m_2)}(l_1l_2)^{\alpha}}{m_1^{\half+it} m_2^{\half-it}} V\lt(\frac{m_1}{x} \rt)V\lt(\frac{m_2}{x} \rt) \label{eq_diag_defn} \\
  S_{o_1} :&= \varphi(Q) G \sum_{l_1,l_2} \conj{\chi(l_1)} \chi(l_2) \sum_{\substack{h,m_2 \geq 1 \\ m_1l_1 = m_2l_2 + hQ}} \frac{A(m_1)\conj{A(m_2)}(l_1l_2)^{\alpha}}{m_1^{\half+it} m_2^{\half-it}} \notag  \allowdisplaybreaks[0] \\
 &\phantom{=\varphi(Q) \sum_{l_1,l_2} \conj{\chi(l_1)}\chi(l_2)} \qqquad \qquad \times  V\lt(\frac{m_1}{x} \rt)V\lt(\frac{m_2}{x} \rt) e^{-G \lt| \log \lt(\tfrac{l_1m_1}{l_2m_2} \rt) \rt|} \label{eq_o1_defn} \\
  S_{o_2} :&= \varphi(Q) G \sum_{l_1,l_2} \conj{\chi(l_1)} \chi(l_2) \sum_{\substack{h,m_1 \geq 1 \\ m_2l_2 = m_1l_1 + hQ}} \frac{A(m_1)\conj{A(m_2)}(l_1l_2)^{\alpha}}{m_1^{\half+it} m_2^{\half-it}} \notag \allowdisplaybreaks[0] \\ 
 &\phantom{=\varphi(Q) \sum_{l_1,l_2} \conj{\chi(l_1)}\chi(l_2)} \qqquad \qquad  \times  V\lt(\frac{m_1}{x} \rt)V\lt(\frac{m_2}{x} \rt) e^{-G \lt| \log \lt(\tfrac{l_1m_1}{l_2m_2} \rt) \rt|} \label{eq_o2_defn}
\end{align}
}%

\begin{remark}
   $S_{o_2} = \conj{S_{o_1}}$.
\end{remark}
At this point, we have converted the problem into studying $S_d$ and $S_{o_1}$.

\section{The diagonal portion $S_d$} \label{sec:diag}
In this section, we focus on analyzing $S_d$. The analysis breaks $S_d$ into two sums, $S_{d_1}$ corresponding to $l_1 = l_2$ and $S_{d_2}$ corresponding to $l_1 \neq l_2$:
\[ S_d = S_{d_1} + S_{d_2}, \]
where
\begin{align}
  S_{d_1} :&= \varphi(Q) G \sum_{l} l^{2\alpha} \sum_{m \geq 1} \frac{|A(m)|^2}{m} V\lt(\frac{m}{x} \rt)V\lt(\frac{m}{x} \rt) \label{eq_S_d1} \\
  S_{d_2} :&= \varphi(Q) G \sum_{\substack{l_1,l_2 \\ l_1 \neq l_2}} \conj{\chi(l_1)} \chi(l_2) \sum_{\substack{m_1,m_2 \geq 1 \\ m_1l_1 = m_2l_2}} \frac{A(m_1)\conj{A(m_2)}(l_1l_2)^{\alpha}}{m_1^{\half+it} m_2^{\half-it}} V\lt(\frac{m_1}{x} \rt)V\lt(\frac{m_2}{x} \rt) \label{eq_d2_defn}
\end{align}

\subsection{The case of $l_1 = l_2$} \label{sec:diag_ez}
For $S_{d_1}$, note that the $m$-sum does not depend on $l$. The contribution of $S_{d_1}$ is as follows:
\begin{proposition} \label{eq_S_d1_prop}
  As $x \to \infty$,
  \begin{equation}
     S_{d_1} = \varphi(Q) G \lt| \sum_l l^{2\alpha} \rt| \frac{(4\pi)^k}{\Gamma(k)}\langle f,f\rangle \log x + O(QG\mathcal{L}^{1 + 2\alpha}) + O(x^{-\varepsilon})
  \end{equation}
\end{proposition}

\begin{proof}
  Applying \eqref{eq_easy_mellin} twice to \eqref{eq_S_d1}, we obtain:
\begin{align}
  S_{d_1} &= \varphi(Q) G \lt| \sum_l l^{2\alpha} \rt| \dblinvmellin{2}{2} \sum_{m \geq 1} \frac{|A(m)|^2}{m^{1+s+w}} x^{s+w} v(s) v(w) \,d  s \,d  w \notag \\
 &= \varphi(Q) G \lt| \sum_l l^{2\alpha} \rt| \dblinvmellin{2}{4} \sum_{m \geq 1} \frac{|A(m)|^2}{m^{1+s}} x^s v(s-w)v(w) \,d  s \,d  w \label{eq_S_d1_int}
\end{align}

Moving the line of integration of $s$ down to $\re s = -\frac{1}{3} - \varepsilon$, we pick up simple poles at $s = 0, s = w$, obtaining:
\begin{equation} \label{eq_S_d1_break}
   S_{d_1} = \res{s=w} S_{d_1} + \res{s=0} S_{d_1} + O(x^{-\varepsilon})
\end{equation}
For the residue at $s = w$, by moving the line of integration $\re w = 2$ down to $-\frac{1}{3}- \varepsilon$, we have:
\begin{align}
   S_{d_1, \res{s=w}} &=  \varphi(Q) G \left| \sum_l l^{2\alpha} \right| \left( \frac{(4\pi)^k}{\Gamma(k)}\langle f,f\rangle \log x + O(1) + O(x^{-\frac{1}{3} - \varepsilon}) \right). \label{eq_S_d1_sw}
\end{align}

We continue with the residue at $s=0$:
\begin{equation} \label{eq_S_d1_s0}
  S_{d_1, \res{s=0}} = \varphi(Q) G \left| \sum_l l^{2\alpha} \right| \frac{(4\pi)^k}{\Gamma(k)} \langle f,f\rangle \left( \invmellin{2} v(-w)v(w) \,d  w \right) 
\end{equation}
This is just $O(QG\mathcal{L}^{1 + 2\alpha})$, upon noting the $l$-sum is $O(\mathcal{L}^{1+2\alpha})$ and the $w$-integral above is a constant.

Now plugging \eqref{eq_S_d1_sw} and \eqref{eq_S_d1_s0} into \eqref{eq_S_d1_break}, we have the proposition once we note that the $l$-sum is $O(\mathcal{L}^{1+2\alpha})$. 
\end{proof}

\subsection{The case of $l_1 \neq l_2$} \label{sec:diag_nez}
In $S_{d_2}$, we have the condition $m_1l_1 = m_2l_2$. When $l_1 \neq l_2$, it implies $m_1 = l_2 m$, $m_2 = l_1m$ for some positive integer $m$. Hence, we have:
\begin{equation} \label{eq_S_d2}
  S_{d_2} = \varphi(Q) G \sum_{\substack{l_1,l_2 \\ l_1 \neq l_2}} \frac{\conj{\chi(l_1)}\chi(l_2)}{l_2^{\half+it} l_1^{\half-it}} (l_1l_2)^{\alpha} \sum_{m \geq 1} \frac{A(l_2m)\conj{A(l_1m)}}{m} V\lt(\frac{l_2m}{x} \rt)V\lt(\frac{l_1m}{x} \rt) 
\end{equation}
Using the same methods to the proof of proposition \ref{eq_S_d1_prop}, we have the following:
\begin{proposition} \label{eq_S_d2_propo}
  As $x \to \infty$,
  \begin{align}
    S_{d_2} = \varphi(Q) G \sum_{\substack{l_1,l_2\\ l_1 \neq l_2}} \frac{\conj{\chi(l_1)}\chi(l_2)}{l_2^{\half+it} l_1^{\half-it}} (l_1l_2)^{\alpha} \frac{(4\pi)^k}{\Gamma(k)}\langle f,f \rangle &E_{l_1,l_2}(1) \log (x/l_2) \notag \\
  &+ O(QG\mathcal{L}^{1+2\alpha+\varepsilon}) + O(x^{-\varepsilon}),
  \end{align}
  where $E_{l_1,l_2}(s)$ is defined as follows:
  \[ E_{l_1,l_2}(s) := \lt( \sum_{m \geq 1} \frac{A(l_2m)\conj{A(l_1m)}}{m^s} \rt) \lt( \sum_{m \geq 1} \frac{|A(m)|^2}{m^s} \rt)^{-1}, \]
  which is essentially a product of ratios of Euler factors at the primes $l_1$ and $l_2$. $E_{l_1,l_2}(s)$ is analytic for $\re s > 0$ and is bounded independent of $l_1,l_2$ in the region.
\end{proposition}

We have now obtained a complete understanding of the diagonal sum, so our next focus is to understand the off-diagonal sums $S_{o_1}$ and $S_{o_2}$. By the remark after \eqref{eq_o2_defn} , it is sufficient for us to understand $S_{o_1}$.

\section{Off-diagonal portion $S_{o_1}$, setting up the integrals} \label{sec:off-d-setup}
Recall from \eqref{eq_o1_defn},
\begin{align}
  S_{o_1} &= \varphi(Q) G \sum_{l_1,l_2} \conj{\chi(l_1)} \chi(l_2) \sum_{\substack{h,m_2 \geq 1 \\ m_1l_1 = m_2l_2 + hQ}} \frac{A(m_1)\conj{A(m_2)}(l_1l_2)^{\alpha}}{m_1^{\half+it} m_2^{\half-it}} \notag  \allowdisplaybreaks[0] \\
 &\phantom{=\varphi(Q) \sum_{l_1,l_2} \conj{\chi(l_1)}\chi(l_2)} \qqquad \qquad \times  V\lt(\frac{m_1}{x} \rt)V\lt(\frac{m_2}{x} \rt) e^{-G \log \lt(\tfrac{l_1m_1}{l_2m_2} \rt)} \label{eq_S_o1_sum}
\end{align}

We will show that this object can be converted into studying a four-fold integral involving the $Z_Q(s,w)$ function:
\begin{proposition} \label{eq_S_o1_int}
As $G \to \infty$, we have
\begin{align}
  S_{o_1} &= \varphi(Q) G \sum_{l_1,l_2} \conj{\chi(l_1)}\chi(l_2) (l_1l_2)^{\half+\alpha} l_1^{it}l_2^{-it} \notag \\ 
  &\quad \times \quadinvmellin{\gamma_1}{\gamma_2}{\gamma_3}{\gamma_4} Z_Q\lt(s, s'-s-\tfrac{k}{2} + 1\rt) \notag \\
  &\phantom{\invmellin{c_1}} \times \betaF{s' - s + \half - \beta}{w+s+\beta + \frac{k}{2}-1-s'+it}{w+\frac{k-1}{2}+it} l_1^{w-\half} l_2^{s'-w} \notag\\
  &\phantom{\invmellin{c_1}} \times x^{s' - \half} v(s'-w)v(w-\half) \frac{\Gamma(\beta)\Gamma(G-\beta)}{\Gamma(G)} \,d  \beta \,d  s' \,d  s \,d  w
\end{align}
where $\re w = \gamma_1 = 1 + 2\theta + 2\varepsilon$, $\re s = \gamma_2 = \frac{3}{4}$, $\re s' = \gamma_3 = \frac{5}{4}$ and $\re \beta = \gamma_4 = \frac{1}{2}+\theta+\varepsilon$.
\end{proposition}

\begin{remark}
  As long as $G$ is chosen such that $G = (1+|t|)^{a} \log^b(Q)$ with $a > 0$ and $b > 4$, then $G \geq 3A = \sqrt{90 \log (Q(1+|t|))}$ for large $Q$ and $t$.
\end{remark}

\begin{proof}
We will focus on the innermost sum of $S_{o_1}$. For convenience, we define:
\[ T_{o_1} := \sum_{\substack{h_0,m_2 \geq 1\\ m_1l_1 = m_2l_2 + h_0Q}} \!\!\!\!\!\! \frac{A(m_1)\conj{A(m_2)}}{m_1^{\half+it}m_2^{\half-it}} V\lt(\frac{m_1}{x} \rt)V\lt(\frac{m_2}{x} \rt) e^{- G \log \lt( \frac{m_1l_1}{m_2l_2} \rt)} \]
Starting with the definition of $T_{o_1}$, we will substitute in $m_1l_1 = m_2l_2 + h_0Q$ in several places:
\[ T_{o_1} = l_1^{\half+it} l_2^{\half-it} \sum_{m_2,h_0} \frac{A(m_1)\conj{A(m_2)}V\lt(\frac{m_1}{x} \rt)V\lt(\frac{m_2}{x} \rt)} {(m_2l_2+h_0Q)^{\half+it}(m_2l_2)^{\half-it}} e^{- G \log \lt( 1 + \frac{h_0Q}{m_2l_2} \rt)} \]

We apply \eqref{eq_easy_mellin} twice, resulting in:
\begin{align}
   T_{o_1} = l_1^{\half+it} l_2^{\half-it} &\dblinvmellin{c_1}{c_2} \! \sum_{m_2,h_0} \!\! \frac{A(m_1) \conj{A(m_2)}}{(m_2l_2+h_0Q)^{\half+w+it} (m_2l_2)^{\half + s-it}} \notag \\
  &\qquad \quad \times v(s)v(w)  x^{s+w} l_2^s l_1^w \lt( 1 + \frac{h_0Q}{m_2l_2} \rt)^{-G}  \,d  s \,d  w  \notag
\end{align}
We quote this identity from \cite{GradsteynRhyzik}:
\begin{equation}
  \invmellin{\gamma} \betaF{u}{\beta-u}{\beta} t^{-u} \,d  u = (1+t)^{-\beta}, \label{eq_neg_bino}
\end{equation}
where $0 < \gamma < \re \beta$.
Manipulating the expression from before and using \eqref{eq_neg_bino}, we end up introducing the $Z_Q$ function defined above:
{
\allowdisplaybreaks
\begin{align*}
  T_{o_1} =  &l_1^{\half+it} l_2^{\half-it} \triinvmellin{c_1}{c_2}{c_3} \! \sum_{m_2,h_0} \!\! \frac{A(m_1) \conj{A(m_2)} \left(1 + \frac{h_0Q}{m_2l_2}\right)^{\frac{k-1}{2}}}{(m_2l_2)^{s+w+1}}  x^{s+w}\allowdisplaybreaks[0] \\
 & \quad \times \left(1+\frac{h_0Q}{m_2l_2}\right)^{-(w+\frac{k}{2}+it)} \!\! v(s)v(w) l_2^s l_1^w \lt( \frac{h_0Q}{ m_2l_2} \rt)^{-\beta} \!\! \frac{\Gamma(\beta)\Gamma(G-\beta)}{\Gamma(G)} \,d  s \,d  w \,d  \beta \\
 = &l_1^{\half+it} l_2^{\half-it} \! \quadinvmellin{c_1}{c_2}{c_3}{c_4} \!\!\!\! Z_Q\lt(s+w+1-u-\beta, u+\beta-\tfrac{k-1}{2}\rt)  \allowdisplaybreaks[0] \\
 & \times \betaF{u}{w+\frac{k}{2}-u+it}{w+\frac{k}{2}+it} v(s)v(w)x^{s+w} l_2^s l_1^w \frac{\Gamma(\beta)\Gamma(G-\beta)}{\Gamma(G)} \,d  u \,d  s \,d  w \,d  \beta
\end{align*}
}%
From here, we will do a series of change of variables. First we change $u \mapsto u  - \beta$:
\begin{align*}
=&l_1^{\half+it} l_2^{\half-it} \quadinvmellin{c'_1}{c'_2}{c'_3}{c'_4} Z_Q(s+w+1-u, u-\tfrac{k-1}{2}) x^{s+w} l_2^s l_1^w  \\
 &\times \betaF{u-\beta}{w+\frac{k}{2}+\beta- u+it}{w+\frac{k}{2}+it} v(s)v(w) \frac{\Gamma(\beta)\Gamma(G-\beta)}{\Gamma(G)} \,d  u \,d  s \,d  w \,d  \beta
\end{align*}
Now we do $s \mapsto s+u-1$:
\begin{align*}
=&l_1^{\half+it} l_2^{\half-it} \! \quadinvmellin{c''_1}{c''_2}{c''_3}{c''_4} \!\!\! Z_Q(s\!+\!w, u-\tfrac{k-1}{2}) x^{s-1+u+w} l_2^{s+u-1} l_1^w  \\
 &\times \betaF{u-\beta}{w+\frac{k}{2}+\beta - u+it}{w+\frac{k}{2}+it}v(s+u-1) v( w) \frac{\Gamma(\beta)\Gamma(G-\beta)}{\Gamma(G)} \,d  u \,d  s \,d  w \,d  \beta
\end{align*}
Then we will change $s \mapsto s-w, u \mapsto u+\frac{k-1}{2}$:
\begin{align*}
 =&l_1^{\half+it} l_2^{\half-it} \! \quadinvmellin{c'''_1}{c'''_2}{c'''_3}{c'''_4} \!\! Z_Q(s, u) x^{s+u+\frac{k-3}{2}} l_2^{s-w+u+\frac{k-3}{2}} l_1^w\\ 
&\times \betaF{u+\frac{k-1}{2}-\beta}{w-u+\half+\beta+it}{w+\frac{k}{2}+it}v(s-w+u+\frac{k-1}{2}-1) v( w) \\
& \times \frac{\Gamma(\beta)\Gamma(G-\beta)}{\Gamma(G)} \,d  u \,d  s \,d  w \,d  \beta
\end{align*}
Finally we will get rid of $u$ and introduce $s' = s + u + \frac{k}{2} - 1$, while also doing $w \mapsto w-\half$:
\begin{align*}
  =&l_1^{\half+it} l_2^{\half-it} \! \quadinvmellin{\gamma_1}{\gamma_2}{\gamma_3}{\gamma_4} \!\!\! Z_Q\lt(s, s'-s-\tfrac{k}{2} + 1\rt) x^{s' - \half}   l_2^{s'-w} l_1^{w-\half}  \\
  &\qqquad \qqquad \qquad \times \betaF{s' - s + \half -\beta}{s + w+\tfrac{k}{2} - 1 - s' +\beta+it}{w+\tfrac{k-1}{2}+it}\\
  &\qqquad \qqquad \qquad \times v(s'-w) v(w-\half) \frac{\Gamma(\beta)\Gamma(G-\beta)}{\Gamma(G)} \,d  \beta \,d  s' \,d  s \,d  w
\end{align*}
This ends the change of variables. We can take the following values for the $\gamma_i$'s: $\re s = 2$, $\re s' = 2 + \frac{k}{2} + \varepsilon$, $\re w = 1 + 2\theta + 2\varepsilon$ and $\re \beta  = 2$.

We can move $\re \beta$ down to $\frac{3}{4}$ without hitting poles. Next, we can move $\re s'$ to $\frac{5}{2} + \varepsilon$ without hitting poles. Now we can move $\re s$ down to $\frac{3}{4}$ without picking up poles. Then we will move $\re s'$ to $\frac{5}{4}$, again without poles. Finally, moving $\re \beta$ to $\frac{1}{2}+\theta+\varepsilon$ does not hit any pole. This proves the proposition. 
\end{proof}

We will separate our analysis of $S_{o_1}$ into 2 parts, the first part corresponding to the discrete spectrum $S_{o_1}^d$ and the second to the continuous spectrum $S_{o_1}^c$. We will also replace $Z_Q(s,s'-s-\tfrac{k}{2} + 1)$ with $\lim_{\delta \to 0} Z_Q(s,s'-s-\tfrac{k}{2} + 1; \delta)$, taking the $\delta$ limit where it is convenient to do so. In our analysis, our top priority is to bring down the effects of $x$; the second priority is to bring down the $t$-contribution. 

\section{Off-diagonal, discrete spectrum}
We will be looking at the growth of the discrete terms in this section:
\begin{align}
  S_{o_1}^d = & \lim_{\delta \to 0}\, \varphi(Q) G \sum_{l_1,l_2} \conj{\chi(l_1)}\chi(l_2) (l_1l_2)^{\half+\alpha} l_1^{it}l_2^{-it} \notag \\
  &\times \quadinvmellin{\gamma_1}{\gamma_2}{\gamma_3}{\gamma_4} \frac{(4\pi)^k 2^{s-\half} (l_1l_2)^{\frac{k-1}{2}}}{\Gamma(s+k-1) 2\sqrt{\pi}}  x^{s'-\half} l_2^{s'-w} l_1^{w-\half}  \notag \\
  &\times  v( s'-w) v( w-\half) \betaF{s' - s + \half -\beta}{s + w+\beta - s' + \tfrac{k}{2} - 1+it}{w+ \tfrac{k-1}{2}+it}  \notag \\
  &\times\sum_{t_j} L_Q(s',\conj{u_j}) M(s,t_j,\delta) \conj{\langle U, u_j \rangle} \frac{\Gamma(\beta)\Gamma(G-\beta)}{\Gamma(G)} \,d  \beta \,d  s' \,d  s \,d  w \label{eq_S_o1d_5int}
\end{align}

Our goal here is twofold: to bring down the $x$-exponent as much as possible and to take the $\delta$-limit. As such, the natural thing to do is to bring $\re s'$. Hence, we move $\re s'$ down to $\re s' = \frac{1}{2} - \varepsilon$, hitting simple poles at $s' = w$ and $s' = s + \beta - \frac{1}{2}$. We break our analysis into bounding these two residues and the moved integral.

\subsection{The residue at $s'=w$}
In this section, we will show the following:
\begin{proposition} \label{eq_S_o1d_Zres_propo}
For $G \asymp (1+|t|)^{\frac{2}{3-2\theta}} \log^5 Q$ and $\mathcal{L} \ll Q$,
  \[ \res{s'=w} S_{o_1}^d \ll G^{1+\varepsilon} Q^{\half+\theta+\varepsilon} \mathcal{L}^{3+2\alpha+\varepsilon} + x^{-\varepsilon}, \]
\end{proposition}

\begin{remark}
In the proof of this proposition, we also justify the reason of having this condition on $G$ .
\end{remark}

\begin{proof}
Taking the residue at $s'=w$, we have:
\begin{align}
  \res{s'=w} S_{o_1}^d = & \lim_{\delta \to 0}\, \varphi(Q) G \sum_{l_1,l_2} \conj{\chi(l_1)}\chi(l_2) (l_1l_2)^{\half+\alpha} l_1^{it}l_2^{-it} \notag \\
  &\times \triinvmellin{\gamma_1}{\gamma_2}{\gamma_4} \frac{(4\pi)^k 2^{s-\half} (l_1l_2)^{\frac{k-1}{2}}}{\Gamma(s+k-1) 2\sqrt{\pi}}  (xl_1)^{w-\half} \notag \\
  &\times  v( w-\half) \betaF{w- s + \half -\beta}{s + \beta + \tfrac{k}{2} - 1+it}{w+ \tfrac{k-1}{2}+it}  \notag \\
  &\times\sum_{t_j} L_Q(w,\conj{u_j}) M(s,t_j,\delta) \conj{\langle U, u_j \rangle} \frac{\Gamma(\beta)\Gamma(G-\beta)}{\Gamma(G)} \,d  \beta \,d  s \,d  w \label{eq_S_o1d_leadsp}
\end{align}
This really suggests moving $\re w$ to $\frac{1}{2}-\varepsilon$, which picks up simple poles at $w=s+\beta-\frac{1}{2}$ and $w=\frac{1}{2}$.

\subsubsection{The residue at $w=s+\beta-\frac{1}{2}$}
We first write down the residue:
\begin{align}
  &\res{w=s+\beta-\frac{1}{2}}\res{s'=w} S_{o_1}^d \notag \\
  = & \lim_{\delta \to 0}\, \varphi(Q) G \sum_{l_1,l_2} \conj{\chi(l_1)}\chi(l_2) (l_1l_2)^{\half+\alpha} l_1^{it}l_2^{-it} \notag \\
  &\times \dblinvmellin{\gamma_2}{\gamma_4} \frac{(4\pi)^k 2^{s-\half} (l_1l_2)^{\frac{k-1}{2}}}{\Gamma(s+k-1) 2\sqrt{\pi}}  (xl_1)^{s+\beta-1} v( s+\beta-1)  \notag \\
  &\times\sum_{t_j} L_Q(s+\beta-\frac{1}{2},\conj{u_j}) M(s,t_j,\delta) \conj{\langle U, u_j \rangle} \frac{\Gamma(\beta)\Gamma(G-\beta)}{\Gamma(G)} \,d  \beta \,d  s \label{eq_S_o1d_leadswp}
\end{align}
Here, we move $\re \beta$ down to $\frac{1}{4}-\varepsilon$, passing through a simple pole at $\beta = 1-s$. The residue is:
\begin{align}
  &\res{\beta=1-s} \res{w=s+\beta-\frac{1}{2}}\res{s'=w} S_{o_1}^d \notag \\
  = & \lim_{\delta \to 0}\, \varphi(Q) G \sum_{l_1,l_2} \conj{\chi(l_1)}\chi(l_2) (l_1l_2)^{\half+\alpha} l_1^{it}l_2^{-it} \sum_{t_j} L_Q(\frac{1}{2},\conj{u_j})  \conj{\langle U, u_j \rangle} \notag \\
  &\times \invmellin{\gamma_2} \frac{(4\pi)^k 2^{s-\half} (l_1l_2)^{\frac{k-1}{2}}}{\Gamma(s+k-1) 2\sqrt{\pi}} M(s,t_j,\delta) \frac{\Gamma(1-s)\Gamma(G-1+s)}{\Gamma(G)} \,ds \label{eq_S_o1d_lead}
\end{align}
Now to take the limit as $\delta \to 0$, it is necessary for us to move $\re s$ to $\frac{1}{2}-\frac{k}{2}-\theta -\varepsilon$, which results in us picking up poles at $s = \frac{1}{2} \pm it_j - r$, for integers $0 \leq r \leq \frac{k}{2}$. For integer $r$, using the notation from proposition \ref{eq_Z_s_poles}  the residue is:
\begin{align}
  &\sum_{\pm t_j} \res{s=\frac{1}{2}+it_j-r} \res{\beta=1-s} \res{w=s+\beta-\frac{1}{2}}\res{s'=w} S_{o_1}^d \notag \\
  = & \varphi(Q) G \sum_{l_1,l_2} \conj{\chi(l_1)}\chi(l_2) (l_1l_2)^{\half+\alpha} l_1^{it}l_2^{-it} \notag \\
  &\qqquad \times \sum_{t_j} (l_1l_2)^{\frac{k-1}{2}} c_{r,j} L_Q(\frac{1}{2},\conj{u_j}) \frac{\Gamma(\frac{1}{2}-it_j+r)\Gamma(G-\frac{1}{2}-it_j+r)}{\Gamma(G)} \label{eq_S_o1d_end1}
\end{align}
Using the fact that $\Gamma(G-\frac{1}{2}-it_j-r) \ll \Gamma(G-\frac{1}{2}+\theta-r)$ and \eqref{eq_L_crj}, we can see that the above expression is bounded by $O(Q^{\frac{1}{2}+\theta+\varepsilon} G^{\frac{1}{2}+\theta-r+\varepsilon} \mathcal{L}^{3+2\alpha +\varepsilon})$.

For the moved integral with $\re s = \gamma_2' = \frac{1}{2}-\frac{k}{2}-\theta-\varepsilon$, we take $\delta \to 0$, obtaining:
\begin{align}
  &\res{\beta=1-s} \res{w=s+\beta-\frac{1}{2}}\res{s'=w} S_{o_1}^d \notag \\
  = & \varphi(Q) G \sum_{l_1,l_2} \conj{\chi(l_1)}\chi(l_2) (l_1l_2)^{\half+\alpha} l_1^{it}l_2^{-it} \invmellin{\gamma_2'} \frac{(4\pi)^k (l_1l_2)^{\frac{k-1}{2}}}{2\Gamma(s+k-1)} \notag \\
  &\quad \times \sum_{t_j} L_Q(\frac{1}{2},\conj{u_j}) \frac{\Gamma(1-s)^2\Gamma(s-\frac{1}{2}+it_j)\Gamma(s-\frac{1}{2}-it_j)}{\Gamma(\frac{1}{2}+it_j)\Gamma(\frac{1}{2}-it_j)} \conj{\langle U, u_j \rangle} \frac{\Gamma(G-1+s)}{\Gamma(G)} \,d  s \label{eq_S_o1d_lead_moved}
\end{align}
By using Stirling's approximation, \eqref{eq_L_inner} and splitting the integral according to the relative sizes of $|t_j|$ and $|\im s|$, one can derive the bound $O(Q^{\frac{1}{2}+\theta+\varepsilon} G^{\frac{1}{2}-\frac{k}{2}-\theta-\varepsilon} \mathcal{L}^{3+2\alpha +\varepsilon})$.

\subsubsection{Residue at $w=\frac{1}{2}$}
The residue is:
\begin{align}
  &\res{w=\frac{1}{2}}\res{s'=w} S_{o_1}^d \notag \\
  = & \lim_{\delta \to 0}\, \varphi(Q) G \sum_{l_1,l_2} \conj{\chi(l_1)}\chi(l_2) (l_1l_2)^{\half+\alpha} l_1^{it}l_2^{-it} \notag \\
  &\times \dblinvmellin{\gamma_2}{\gamma_4} \frac{(4\pi)^k 2^{s-\half} (l_1l_2)^{\frac{k-1}{2}}}{\Gamma(s+k-1) 2\sqrt{\pi}} \sum_{t_j} L_Q(\frac{1}{2},\conj{u_j}) M(s,t_j,\delta) \conj{\langle U, u_j \rangle} \notag \\
  &\times  \betaF{1- s -\beta}{s + \beta + \tfrac{k}{2} - 1+it}{\tfrac{k}{2}+it} \frac{\Gamma(\beta)\Gamma(G-\beta)}{\Gamma(G)} \,d  \beta \,d  s \label{eq_S_o1d_2ndw}
\end{align}
Similar as before, we need to move $\re s$ to $\frac{1}{2}-\frac{k}{2}-\theta-\varepsilon$ to take the limit. In doing so, we pick up simple poles at $s = 1-\beta$, $s=1-\beta-\frac{k}{2}-it$ and $s=\frac{1}{2}\pm it_j-r$ for $0 \leq r \leq \frac{k}{2}$.

Investigating the pole at $s=1-\beta$, the residue is:
\begin{align}
  &\res{s=1-\beta} \res{w=\frac{1}{2}}\res{s'=w} S_{o_1}^d \notag \\
  = &- \lim_{\delta \to 0}\, \varphi(Q) G \sum_{l_1,l_2} \conj{\chi(l_1)}\chi(l_2) (l_1l_2)^{\half+\alpha} l_1^{it}l_2^{-it} \sum_{t_j} L_Q(\frac{1}{2},\conj{u_j}) \conj{\langle U, u_j \rangle} \notag \\
  &\times \invmellin{\gamma_4} \frac{(4\pi)^k 2^{\half-\beta} (l_1l_2)^{\frac{k-1}{2}}}{\Gamma(k-\beta) 2\sqrt{\pi}} M(1-\beta,t_j,\delta) \frac{\Gamma(\beta)\Gamma(G-\beta)}{\Gamma(G)} \,d\beta \label{eq_S_o1d_2ndws}
\end{align}
Changing variable $\beta = 1-s$, we see that the contour line becomes $\re s = \frac{1}{2}-\theta-\varepsilon$. This is pretty much the same term as \eqref{eq_S_o1d_lead}. Thus, the same bound applies.

Next up is the residue at $s=1-\beta-\frac{k}{2}-it$, for which we also take $\delta \to 0$:
\begin{align}
  &\res{s=1-\beta-\frac{k}{2}-it} \res{w=\frac{1}{2}}\res{s'=w} S_{o_1}^d \notag \\
  = & \varphi(Q) G \sum_{l_1,l_2} \conj{\chi(l_1)}\chi(l_2) (l_1l_2)^{\half+\alpha} l_1^{it}l_2^{-it} \notag \\
  &\qqquad \times \invmellin{\gamma_4} \frac{(4\pi)^k(l_1l_2)^{\frac{k-1}{2}}}{2\Gamma(\frac{k}{2}-\beta-it) }  \frac{\Gamma(\beta)\Gamma(G-\beta)}{\Gamma(G)} \sum_{t_j} L_Q(\frac{1}{2},\conj{u_j})\conj{\langle U, u_j \rangle}\notag \\
  &\qqquad \times  \frac{\Gamma(\beta+\frac{k}{2}+it)\Gamma(\frac{1}{2}-\beta-\frac{k}{2}-it+it_j)\Gamma(\frac{1}{2}-\beta-\frac{k}{2}-it-it_j)}{\Gamma(\frac{1}{2}+it_j)\Gamma(\frac{1}{2}-it_j)} \,d  \beta \label{eq_S_o1d_2ndwLs}
\end{align}
Changing variable $\beta \mapsto \beta - it$, the expression becomes:
\begin{align}
  &\res{s=1-\beta-\frac{k}{2}-it} \res{w=\frac{1}{2}}\res{s'=w} S_{o_1}^d \notag \\
  = & \varphi(Q) G \sum_{l_1,l_2} \conj{\chi(l_1)}\chi(l_2) (l_1l_2)^{\half+\alpha} l_1^{it}l_2^{-it} \notag \\
  &\qqquad \times \invmellin{\gamma_4} \frac{(4\pi)^k(l_1l_2)^{\frac{k-1}{2}}}{2\Gamma(\frac{k}{2}-\beta-it) }  \frac{\Gamma(\beta-it)\Gamma(G-\beta+it)}{\Gamma(G)} \sum_{t_j} L_Q(\frac{1}{2},\conj{u_j})\conj{\langle U, u_j \rangle}\notag \\
  &\qqquad \times  \frac{\Gamma(\beta+\frac{k}{2})\Gamma(\frac{1}{2}-\beta-\frac{k}{2}+it_j)\Gamma(\frac{1}{2}-\beta-\frac{k}{2}-it_j)}{\Gamma(\frac{1}{2}+it_j)\Gamma(\frac{1}{2}-it_j)} \,d  \beta \label{eq_S_o1d_2ndwLs_break}
\end{align}
Note that there is enough exponential decay in $\beta$ even if we sum $t_j$ first. Hence, we can easily see that this term is bounded by $O(G^{\frac{1}{2}+\varepsilon} Q^{\frac{1}{2}+\theta+\varepsilon} \mathcal{L}^{3+2\alpha+\varepsilon})$.

We now keep track of the residues at $s=\frac{1}{2}+it_j-r$:
\begin{align}
  &\sum_{\pm t_j} \res{s=\frac{1}{2}+it_j-r} \res{w=\frac{1}{2}}\res{s'=w} S_{o_1}^d \notag \\
  = & \varphi(Q) G \sum_{l_1,l_2} \conj{\chi(l_1)}\chi(l_2) (l_1l_2)^{\half+\alpha} l_1^{it}l_2^{-it} \invmellin{\gamma_4} \sum_{t_j} (l_1l_2)^{\frac{k-1}{2}} c_{r,j} L_Q(\frac{1}{2},\conj{u_j})  \notag \\
  &\times  \betaF{\frac{1}{2}-it_j+r -\beta}{\beta + \tfrac{k-1}{2} +it_j-r+it}{\tfrac{k}{2}+it} \frac{\Gamma(\beta)\Gamma(G-\beta)}{\Gamma(G)} \,d  \beta \,d  s \label{eq_dis_res_rj}
\end{align}

Upon changing variable $\beta \mapsto \beta-it_j$, the expression above is related to the upcoming lemma.

\begin{lemma} \label{propo_crazy}
For $\re \beta = a \geq \frac{1}{2}+r+\varepsilon$ and $\mathcal{L} \ll Q$,
\begin{align}
  & G \sum_{t_j} L_Q(\half,\conj{u_j}) c_{r,j} \notag \\
  &\qquad \times  \invmellin{a} \!\! \betaF{\half + r -\beta}{\beta+\tfrac{k-1}{2} -r+i t}{\tfrac{k}{2}+it} \frac{\Gamma(\beta-it_j) \Gamma(G-\beta+it_j)}{\Gamma(G)} \,d  \beta \notag \\
  \ll & G^{1-a + \theta} \lt( (1+|t|)^{a-r-\half+\varepsilon} + (1+|t|)^{1 + r+ \varepsilon} \rt) Q^{-\half+\theta} \mathcal{L}^{1-k+\varepsilon} \label{eq_dis_res_bound}
\end{align}
In fact, to minimize this, the optimal choices are $\re \beta = a = \frac{3}{2} + 2r + \varepsilon$ and $G \asymp (1+|t|)^\frac{2}{3-2\theta} \log^5 Q$. With these choices, the bound is $O(G^{1+\varepsilon} Q^{-\half+\theta} \mathcal{L}^{1-k+\varepsilon})$.
\end{lemma}
This relates to \eqref{eq_dis_res_rj} by moving lines of integration and picking up relevent poles. We will delay the proof of this lemma to the end of this section.

Going back to the integral in \eqref{eq_dis_res_rj}, we will deal with the case $r=0$ first. We move the line of integration up to $\re \beta = \frac{3}{2} + \varepsilon$, picking up a residue at $\beta = \frac{3}{2}$. Then by lemma \ref{propo_crazy} with the optimal choice of $G$, the moved integral is $O(G^{1+\varepsilon} Q^{\half+\theta} \mathcal{L}^{3 + 2\alpha + \varepsilon})$. The residue is:
\begin{align}
  - \varphi(Q) G \sum_{l_1,l_2} \conj{\chi(l_1)}\chi(l_2) (l_1l_2)^{\frac{k}{2}+\alpha} l_1^{it}&l_2^{-it} (\tfrac{k}{2}+it) \notag \\
   &\times \sum_{t_j} L_Q(\half,\conj{u_j}) c_{0,j}  \frac{\Gamma(\tfrac{3}{2}-it_j) \Gamma(\beta-\frac{3}{2}+it_j)}{\Gamma(G)} \notag
\end{align}
With the same choice of $G$ as above, this residue is also $O(G^{1+\varepsilon} Q^{\half+\theta} \mathcal{L}^{3+2\alpha+\varepsilon})$. We assume the same choice of $G$ from now on.

For $1 \leq r \leq \frac{k}{2}$, we move the line of integration in \eqref{eq_dis_res_rj} up to $\re \beta = \frac{3}{2} + 2r + \varepsilon$. The moved integral is $O(G^{1+\varepsilon} Q^{\half+\theta} \mathcal{L}^{3+2\alpha+\varepsilon})$ as shown in lemma \ref{propo_crazy}. The residues for a particular $r$ are as follows:
\begin{align}
  &\varphi(Q) G \sum_{l_1,l_2} \conj{\chi(l_1)}\chi(l_2) (l_1l_2)^{\frac{k}{2}+\alpha} l_1^{it}l_2^{-it} L_Q(\half,\conj{u_j}) c_{r,j} \notag \\
  &\times \sum_{m=0}^{r+1} \frac{(-1)^m}{m!} \frac{\Gamma(\tfrac{k}{2}+m+it)}{\Gamma(\tfrac{k}{2}+it)} \frac{\Gamma(\half+r+m-it_j)\Gamma(G-\frac{1}{2}-r-m+it_j)}{\Gamma(G)} \notag
\end{align}
This whole sum is $o(GQ^{\half+\theta}\mathcal{L}^{3+2\alpha+\varepsilon})$.

Last but not least, we bound the moved integral, where $\re s = \gamma_2'=\frac{1}{2}-\frac{k}{2}-\theta-\varepsilon$:
\begin{align}
  &\res{w=\frac{1}{2}}\res{s'=w} S_{o_1}^d \notag \\
  = & \lim_{\delta \to 0}\, \varphi(Q) G \sum_{l_1,l_2} \conj{\chi(l_1)}\chi(l_2) (l_1l_2)^{\half+\alpha} l_1^{it}l_2^{-it} \notag \\
  &\times \dblinvmellin{\gamma_2'}{\gamma_4} \frac{(4\pi)^k 2^{s-\half} (l_1l_2)^{\frac{k-1}{2}}}{\Gamma(s+k-1) 2\sqrt{\pi}} \sum_{t_j} L_Q(\frac{1}{2},\conj{u_j}) M(s,t_j,\delta) \conj{\langle U, u_j \rangle} \notag \\
  &\times  \betaF{1- s -\beta}{s + \beta + \tfrac{k}{2} - 1+it}{\tfrac{k}{2}+it} \frac{\Gamma(\beta)\Gamma(G-\beta)}{\Gamma(G)} \,d  \beta \,d  s \label{eq_dis_moved}
\end{align}
Note that for $\re s < \half - \frac{k}{2}$, we note the following bound which is true regardless of relative sizes of $|t_j|$ and $|\im s|$:
\[ \lim_{\delta \to 0} \frac{2^{s-\half}M(s,t_j,\delta)}{\Gamma(s+k-1) 2\sqrt{\pi}}  \ll (1+|t_j|)^{2\re s - 2} (1+|\im s|)^{3-3\re s - k} \]
Hence, at $\re s = \half - \frac{k}{2} - \theta - \varepsilon$, using the above bound and \eqref{eq_L_inner},we have:
\begin{align*}
    \lim_{\delta \to 0} \frac{(4\pi)^k 2^{s-\half}}{\Gamma(s+k-1) 2\sqrt{\pi}} &\sum_j L_Q(\half,\conj{u_j}) M(s,t_j,\delta) \conj{\langle U, u_j \rangle} \notag \\
   &\ll  (1+|\im s|)^{\frac{3}{2} +\frac{k}{2} + 3\theta +3\varepsilon} Q^{-\half+\theta} \mathcal{L}^{1-k+\varepsilon}
\end{align*}

Using the fact noted above and the methods of proving lemma \ref{propo_crazy}, we have the following auxiliary lemma:
\begin{lemma} \label{propo_crazy2}
For  $G \asymp (1+|t|)^{\frac{2}{3-2\theta}} \log^5 Q$ and $\re \beta = a = \frac{5}{2} + k + 5\theta + 2\varepsilon$,
\begin{align}
  &G \dblinvmellin{\gamma'_2}{a} \sum_{t_j} \lim_{\delta \to 0} \frac{(4\pi)^k 2^{s-\half}M(s,t_j,\delta)}{\Gamma(s+k-1) 2\sqrt{\pi}}  L_Q(\half,\conj{u_j})  \conj{\langle U, u_j \rangle} \notag \\
  & \qqquad \qqquad \quad \times \betaF{1 - s - \beta}{s+\beta + \tfrac{k}{2} - 1 +it}{\tfrac{k}{2}+it} \frac{\Gamma(\beta)\Gamma(G-\beta)}{\Gamma(G)} \,d  \beta \,d  s \notag \\
  \ll & G^{1+\varepsilon} Q^{-\half+\theta} \mathcal{L}^{1-k+\varepsilon} \label{eq_dis_moved_bound}
\end{align}
\end{lemma}
Again in \eqref{eq_dis_moved}, we move $\re \beta$ to $\frac{5}{2} + k + 5\theta + 2\varepsilon$, hitting poles at $\beta = 1 - s + \ell$, where $0 \leq \ell \leq \frac{k}{2} + 2$. Using the proposition above and $G \asymp (1+|t|)^{\frac{2}{3-2\theta}} \log^3 Q$, the moved integral is $O(G^{1+\varepsilon} Q^{\half+\theta+\varepsilon} \mathcal{L}^{3+2\alpha+\varepsilon})$. The residue at $\beta = 1-s + \ell$ is:
\begin{align}
  \lim_{\delta \to 0} \varphi(Q) G \sum_{l_1,l_2} &\conj{\chi(l_1)}\chi(l_2) (l_1l_2)^{\frac{k}{2}+\alpha} l_1^{it}l_2^{-it} \notag \\
  &\times \invmellin{\gamma'_2} \sum_j \frac{(4\pi)^k 2^{s-\half}M(s,t_j,\delta)}{\Gamma(s+k-1) 2\sqrt{\pi}} L_Q(\half,\conj{u_j})  \conj{\langle U, u_j \rangle} \notag \\
  & \qquad \qquad \times \frac{(-1)^\ell}{\ell !} \frac{\Gamma(\tfrac{k}{2} + \ell +it)}{\Gamma(\tfrac{k}{2}+it)} \frac{\Gamma(1-s+\ell)\Gamma(G+s-1-\ell)}{\Gamma(G)} \,d  s \notag
\end{align}
These residues are $o(G Q^{\half+\varepsilon} \mathcal{L}^{3+2\alpha+\varepsilon} )$. Putting all these together, we see the residue at $w=\frac{1}{2}$ is bounded by $O(G Q^{\half+\varepsilon} \mathcal{L}^{3+2\alpha+\varepsilon} )$.

\subsubsection{The moved integral at $w=\frac{1}{2}-\varepsilon$}
We see that the $x$-exponent is $O(x^{-\varepsilon})$. Moving the line of integration for $\re s$ such that we can take the $\delta$-limit just as demonstrated before, we see that this term is $O(x^{-\varepsilon})$.

Putting all the cases so far together, we finish proving proposition \ref{eq_S_o1d_Zres_propo}.

\end{proof}

\subsection{Contribution from the pole $s'=s+\beta-\frac{1}{2}$}
\begin{proposition} \label{eq_S_o1d_Z_sp2}
For $G \asymp (1+|t|)^{\frac{2}{3-2\theta}} \log^5 Q$,
  \[ \res{s'=s+\beta-\frac{1}{2}} S_{o_1}^d \ll x^{-\varepsilon}, \]
\end{proposition}

\begin{proof}
  We start by writing down the residue at $s'=s+\beta-\frac{1}{2}$:
  \begin{align}
  &\res{s'=s+\beta-\frac{1}{2}} S_{o_1}^d \notag \\
  = & \lim_{\delta \to 0}\, \varphi(Q) G \sum_{l_1,l_2} \conj{\chi(l_1)}\chi(l_2) (l_1l_2)^{\half+\alpha} l_1^{it}l_2^{-it} \notag \\
  &\!\!\times \triinvmellin{\gamma_1}{\gamma_2}{\gamma_4} \frac{(4\pi)^k 2^{s-\half} (l_1l_2)^{\frac{k-1}{2}}}{\Gamma(s+k-1) 2\sqrt{\pi}}  x^{s+\beta-1} l_2^{s+\beta-\frac{1}{2}-w} l_1^{w-\half} v( w-\tfrac{1}{2}) \notag \\
  &\!\!\times  v( s+\beta-\tfrac{1}{2}-w)  \sum_{t_j} L_Q(s+\beta-\tfrac{1}{2},\conj{u_j}) M(s,t_j,\delta) \conj{\langle U, u_j \rangle} \frac{\Gamma(\beta)\Gamma(G-\beta)}{\Gamma(G)} \,d  \beta \,d  s \,d  w \label{eq_S_o1d_sp2_first}
\end{align}
We move $\re \beta$ down to $\frac{1}{4}-\varepsilon$, during which we do not encounter any poles. Hence, this term is $O(x^{-\varepsilon})$, after moving $\re s$ to $\frac{1}{2}-\frac{k}{2}-\theta-\varepsilon$ to enable us taking $\delta \to 0$.
\end{proof}

\subsection{Contribution of the moved integral at $\re s' = \frac{1}{2}-\varepsilon$}
We see that the $x$-exponent is negative. Moving $\re s$ down to enable us taking the $\delta$-limit, this term is also $O(x^{-\varepsilon})$.

\subsection{Proof of lemma \eqref{propo_crazy}}
For convenience, we write $\beta = a + it_\beta$, where $a, t_\beta \in \mathbb{R}$. We will split the object to be analyzed as follows:
\begin{align}
  & G \sum_{t_j} L_Q(\half,\conj{u_j}) c_{r,j} \notag \\
  &\qquad \times  \invmellin{a} \!\! \betaF{\half + r -\beta}{\beta+\tfrac{k-1}{2} -r+i t}{\tfrac{k}{2}+it} \frac{\Gamma(\beta-it_j)\Gamma(G-\beta+it_j)}{\Gamma(G)} \,d  \beta \notag \\
  =& G \sum_{t_j} L_Q(\half,\conj{u_j}) c_{r,j} (P_1 + P_2 + P_3), \notag
\end{align}
where
\begin{align}
  P_1 &= \frac{1}{2\pi i} \!\!\!\!\!\!\!\! \int \limits_{\substack{\beta = a + it_\beta \\ |t_\beta| \leq |t| - \log^4 |t|}} \!\!\!\!\!\!\!\!\!\! \betaF{\half + r -\beta}{\beta+\tfrac{k-1}{2} -r+i t}{\tfrac{k}{2}+it} \frac{\Gamma(\beta-it_j)\Gamma(G-\beta+it_j)}{\Gamma(G)} \,d  \beta \notag \\
  P_2 &= \frac{1}{2\pi i} \!\!\!\!\!\!\!\! \int \limits_{\substack{\beta = a + it_\beta \\ \lt| |t_\beta| - |t| \rt| \leq \log^4 |t|}} \!\!\!\!\!\!\!\!\!\! \betaF{\half + r -\beta}{\beta+\tfrac{k-1}{2} -r+i t}{\tfrac{k}{2}+it} \frac{\Gamma(\beta-it_j)\Gamma(G-\beta+it_j)}{\Gamma(G)} \,d  \beta \notag \\
  P_3 &= \frac{1}{2\pi i} \!\!\!\!\!\!\!\! \int \limits_{\substack{\beta = a + it_\beta \\ |t_\beta| \geq |t| + \log^4 |t|}} \!\!\!\!\!\!\!\!\!\! \betaF{\half + r -\beta}{\beta+\tfrac{k-1}{2} -r+i t}{\tfrac{k}{2}+it} \frac{\Gamma(\beta-it_j)\Gamma(G-\beta+it_j)}{\Gamma(G)} \,d  \beta \notag
\end{align}
In each part, we seek to bound the integrand using Stirling's formula.
\subsubsection{The case where $|t_\beta| \leq |t| - \log^4 |t|$: $P_1$}
The ratio of gammas is bounded by:
\begin{align}
  &\betaF{\half + r -\beta}{\beta+\tfrac{k-1}{2} -r+i t}{\tfrac{k}{2}+it} \frac{\Gamma(\beta-it_j)\Gamma(G-\beta+it_j)}{\Gamma(G)} \notag \\
\ll &G^{-a+\theta} (1 + |t|)^{a-r-\half} (1+|t_\beta|)^{r-a} (1+|t_\beta - t_j|)^{a+\theta-\half} e^{-\frac{\pi}{2} |t_\beta-t_j|} \notag
\end{align}
We further separate this case by the relative sizes of $|t_j|$ and $|t|$:
\begin{enumerate}[leftmargin=*]
\item If $|t_j| \geq |t|$, then we can further conclude that the integrand is bounded by:
\[ (1 + |t|)^{a-r-\half} (1+|t_\beta|)^{r-a} (1+|t_j|)^{a+\theta-\half} e^{-\frac{\pi}{2} |t_j|} e^{\frac{\pi}{2} |t_\beta|} G^{-a+\theta} \]
Executing the $\beta$-integral and then summing over such $|t_j|$'s using the bound in proposition \ref{propo_j_sums}, we have that:
\[ G  \sum_{|t_j| \geq |t|}  L_Q(\half,\conj{u_j}) c_{r,j} P_1 \ll G^{1-a+\theta} (1 + |t|)^{a-r-\half+\varepsilon} Q^{-\half} L^{1-k+\varepsilon} \]
\vspace{0.4em}
\item If $|t_j| \leq |t|$, then we separate the integral as follows:
  \[ P_1 = P_{1,1} + P_{1,2} + P_{1,3}, \]
  where
  \begin{align*}
    P_{1,1} &= \frac{1}{2\pi i} \!\!\!\!\!\!\!\! \int \limits_{\substack{\beta = a + it_\beta \\ |t_\beta| \leq |t_j| - \log^4 |t_j|}} \!\!\!\!\!\!\!\!\!\! \betaF{\half + r -\beta}{\beta+\tfrac{k-1}{2} -r+i t}{\tfrac{k}{2}+it} \frac{\Gamma(\beta-it_j)\Gamma(G-\beta+it_j)}{\Gamma(G)} \,d  \beta \notag \\
    P_{1,2} &= \frac{1}{2\pi i} \!\!\!\!\!\!\!\! \int \limits_{\substack{\beta = a + it_\beta \\ \lt| |t_\beta| - |t_j| \rt| \leq \log^4 |t_j|}} \!\!\!\!\!\!\!\!\!\! \betaF{\half + r -\beta}{\beta+\tfrac{k-1}{2} -r+i t}{\tfrac{k}{2}+it} \frac{\Gamma(\beta-it_j)\Gamma(G-\beta+it_j)}{\Gamma(G)} \,d  \beta \notag \\
    P_{1,3} &= \frac{1}{2\pi i} \!\!\!\!\!\!\!\! \int \limits_{\substack{\beta = a + it_\beta \\ |t_\beta| \geq |t_j| + \log^4 |t_j| \\ |t_\beta| \leq |t| - \log^4 |t|}} \!\!\!\!\!\!\!\!\!\! \betaF{\half + r -\beta}{\beta+\tfrac{k-1}{2} -r+i t}{\tfrac{k}{2}+it} \frac{\Gamma(\beta-it_j)\Gamma(G-\beta+it_j)}{\Gamma(G)} \,d  \beta \notag
  \end{align*}
  \begin{enumerate}[leftmargin=*]
    \item In the subcase $|t_\beta| \leq |t_j| - \log^4 |t_j|$, this subcase has the same effect as case 1:
      \[ G \sum_{|t_j| \leq |t|} L_Q(\half,\conj{u_j}) c_{r,j} P_{1,1} \ll G^{1-a+\theta} (1 + |t|)^{a-r-\half+\varepsilon} Q^{-\half} L^{1-k+\varepsilon} \]
    \item In the subcase $\lt| |t_\beta| - |t_j| \rt| \leq \log^4 |t_j|$, the integrand is bounded by
      \[ (1 + |t|)^{a-r-\half+\varepsilon} (1+|t_\beta|)^{r-a} G^{-a+\theta}. \]
      Now executing the $\beta$-integral and then summing $|t_j|$, we obtain that
      \[ G \sum_{|t_j| \leq |t|} L_Q(\half,\conj{u_j}) c_{r,j} P_{1,2} \ll G^{1-a+\theta} (1 + |t|)^{1+r+\varepsilon} Q^{-\half} L^{1-k+\varepsilon} \]
    \item In the last subcase, $|t_j| + \log^4 |t_j| \leq |t_\beta| \leq |t| - \log^4 |t|$, the integrand is bounded by:
      \[ (1 + |t|)^{a-r-\half} (1+|t_\beta|)^{r+\theta-\half} e^{-\frac{\pi}{2} |t_\beta|} e^{\frac{\pi}{2}|t_j|} G^{-a+\theta}.\]
      Now executing the $\beta$-integral and then summing $|t_j|$, we obtain that
     \[ G \sum_{|t_j| \leq |t|} L_Q(\half,\conj{u_j}) c_{r,j} P_{1,3} \ll G^{1-a+\theta} (1 + |t|)^{a-r-\half+\varepsilon} Q^{-\half} L^{1-k+\varepsilon} \]
  \end{enumerate}
\end{enumerate}
Hence, in total:
\begin{align}
  G \sum_{j}  L_Q(\half,\conj{u_j}) &c_{r,j} P_1 \notag \\
   &\ll G^{1-a+\theta} \lt( (1 + |t|)^{1+r+\varepsilon} + (1+|t|)^{a-r-\half+\varepsilon} \rt) Q^{-\half} L^{1-k+\varepsilon} \notag
\end{align}

\subsubsection{The case where $\lt| |t_\beta| - |t| \rt| \leq \log^4 |t|$: $P_2$}
The ratio of gammas is bounded by:
\begin{align}
  &\betaF{\half + r -\beta}{\beta+\tfrac{k-1}{2} -r+i t}{\tfrac{k}{2}+it} \frac{\Gamma(\beta-it_j)\Gamma(G-\beta+it_j)}{\Gamma(G)}\notag \\
&\qqqquad \ll G^{-a+\theta} (1 + |t|)^{r-a-\frac{k-1}{2}} (1+|t_\beta - t_j|)^{a+\theta-\half} e^{-\frac{\pi}{2} |t_\beta-t_j|} \notag
\end{align}
Separating cases by the relative sizes of $|t_j|$ and $|t|$, we can show that this term is completely overshadowed by $P_1$.

\subsubsection{The case where $|t_\beta| \geq |t| + \log^4 |t|$: $P_3$}
The ratio of gammas is bounded by:
\begin{align}
  &\betaF{\half + r -\beta}{\beta+\tfrac{k-1}{2} -r+i t}{\tfrac{k}{2}+it} \frac{\Gamma(\beta-it_j)\Gamma(G-\beta+it_j)}{\Gamma(G)} \notag \\
\ll &G^{-a+\theta} (1 + |t|)^{-\frac{k-1}{2}} (1+|t_\beta|)^{\frac{k}{2}-1} e^{-\pi|t_\beta|} e^{\pi|t|} (1+|t_\beta - t_j|)^{a+\theta-\half} e^{-\frac{\pi}{2} |t_\beta-t_j|} \notag
\end{align}
Executing the $\beta$-integral will show that the result will have decay in $|t|$ that is faster than every polynomial. Hence, this case is also negligible.

Putting the cases together, we realize that:
\begin{align}
  & G \sum_{t_j} L_Q(\half,\conj{u_j}) c_{r,j} \notag \\
  &\qquad \times  \invmellin{a} \!\! \betaF{\half + r -\beta}{\beta+\tfrac{k-1}{2} -r+i t}{\tfrac{k}{2}+it} \frac{\Gamma(\beta-it_j)\Gamma(G-\beta+it_j)}{\Gamma(G)} \,d  \beta \notag \\
  \ll&  G^{1-a+\theta} \lt( (1 + |t|)^{1+r+\varepsilon} + (1+|t|)^{a-r-\half+\varepsilon} \rt) Q^{-\half} L^{1-k+\varepsilon} \notag
\end{align}
Our goal is to minimize $G$. In order to do this, we want to increase $a$, looking at the term $(1+|t|)^{1+r+\varepsilon}$, which has a fixed exponent. However, the other term dominates if $a$ is too large, with increasingly worse behavior as $a$ increases. Hence, we set the exponents of the terms to equal each other. This gives $a = \frac{3}{2} + 2r + \varepsilon$ (the $\varepsilon$ is added in to avoid the pole on the line $\re \beta = \frac{3}{2} + 2r$.)

Finally, we can proceed to get a lower bound of $G$ such that the above is $O(G^{1+\varepsilon} Q^{-\half} L^{1-k+\varepsilon})$, such that the contribution in $(1+|t|)$-aspect is at most as much as in $S_d$. It turns out that $G \asymp (1+|t|)^{\frac{2}{3-2\theta}} \log^5 Q$ when $r=0$. The $G$ required for larger values of $r$ has a smaller $(1+|t|)$-exponent. This proves the proposition.

\section{Off-diagonal, continuous spectrum}
Quoting the continuous spectrum from \eqref{eq_Z_decomp}, the object to analyze here is:
\begin{align}
  S_{o_1}^c = &\lim_{\delta \to 0}  \varphi(Q) G \sum_{l_1,l_2} \conj{\chi(l_1)}\chi(l_2) (l_1l_2)^{\frac{k}{2}+\alpha} l_1^{it}l_2^{-it} \notag \\
  &\times \quadinvmellin{\gamma_1}{\gamma_2}{\gamma_3}{\gamma_4} \frac{(4\pi)^k 2^{s-\half}}{\Gamma(s+k-1) 2\sqrt{\pi}} \notag \\
  &\times \mathcal{V}_{N[l_1,l_2]} \sum_\mathfrak{a} \invmellin{0} \zeta_{\mathfrak{a},Q}(s',-z) M(s,\tfrac{z}{i},\delta)  \conj{ \langle U,E_\mathfrak{a}(*,\half+z) \rangle}  \notag \\
  &\times  \betaF{s' - s + \half - \beta}{s + w+ \beta - s' + \tfrac{k}{2} - 1+it}{w+ \tfrac{k-1}{2}+it}  \notag \\
  &\times v( s'-w) v( w-\half) x^{s'-\half} l_2^{s'-w} l_1^{w-\half} \frac{\Gamma(\beta)\Gamma(G-\beta)}{\Gamma(G)} \,d  z \,d  \beta \,d  s' \,d  s \,d  w \label{eq_S_o1c_6int}
\end{align}
We will move $\re z$ slightly to the right to a curve $C$, which has the property that if $z$ is any complex number between 0 and $C$, $\zeta^*(1-2z) \neq 0$.

For the most part, there will be a lot of similiarities in how we analyze this 5-fold integral compared to how we analyze the discrete spectrum expression in the last section. In particular, we start by moving $\re s'$ down to $\frac{1}{2}-\varepsilon$, passing through simple poles at $s'=w$, $s'=1+z$, $s'=1-z$ and $s'=s+\beta-\frac{1}{2}$. As seen in the case of discrete spectrum, the moved integral here is $O(x^{-\varepsilon})$.

\subsection{The residue at $s'=w$}
This is the pole that has most of the contribution from the continuous part.
\begin{proposition} \label{propo_total_cts}
For $G \asymp (1+|t|)^{\frac{2}{3-2\theta}} \log^5 Q$,
\begin{align}
 \res{s'=w} S_{o_1}^c =&-\frac{\varphi(Q) G}{2 \Gamma(k)} \sum_{l_1,l_2} \conj{\chi(l_1)}\chi(l_2) (l_1l_2)^{\frac{1}{2}+\alpha} l_1^{it}l_2^{-it}  (4\pi)^k \langle f,f \rangle \log x  \cdot  b_{l_1,l_2}  \notag \\
  &+ O(G^{1+\varepsilon} Q^{1+\varepsilon} \mathcal{L}^{1+2\alpha+\varepsilon}) + O(G^{1+\varepsilon} Q^{\half+\varepsilon} \mathcal{L}^{3 + 2\alpha + \varepsilon}) + O(x^{-\varepsilon}), \label{prop_total_cts}
\end{align}
where
\begin{align}
  b_{l_1,l_2} = \begin{cases} l_1^{-1} & \text{ if $l_1 =  l_2$} \\ (l_1l_2)^{-1} E_{l_1,l_2}(1) & \text{ if $l_1 \neq l_2$} \end{cases} \label{eq_defn_b_l1l2}
\end{align}
\end{proposition}
\begin{remark}
  Note the sum of the $(\log x)$-terms above is exactly $-\half$ of the $(\log x)$-terms in $S_d$.
\end{remark}

\begin{proof}
  We first write down the residue at $s'=w$:
  \begin{align}
  \res{s'=w} S_{o_1}^c 
  = &\lim_{\delta \to 0}  \varphi(Q) G \sum_{l_1,l_2} \conj{\chi(l_1)}\chi(l_2) (l_1l_2)^{\frac{k}{2}+\alpha} l_1^{it}l_2^{-it} \notag \\
  &\times \quadinvmellin{\gamma_1}{\gamma_2}{\gamma_4}{C} \frac{(4\pi)^k 2^{s-\half}}{\Gamma(s+k-1) 2\sqrt{\pi}} \frac{\Gamma(\beta)\Gamma(G-\beta)}{\Gamma(G)} \notag \\
  &\times \mathcal{V}_{N[l_1,l_2]} M(s,\tfrac{z}{i},\delta)  \betaF{w - s + \half - \beta}{s + \beta+ \tfrac{k}{2} - 1+it}{w+ \tfrac{k-1}{2}+it}  \notag \\
  &\times v( w-\half) (xl_1)^{w-\half} \sum_\mathfrak{a} \zeta_{\mathfrak{a},Q}(w,-z)   \conj{ \langle U,E_\mathfrak{a}(*,\half+z) \rangle} \,d  z \,d  \beta \,d  s \,d  w \label{eq_S_o1c_sp_lead}
  \end{align}
  At this point, it is natural to move $\re w$ to $\frac{1}{2}-\varepsilon$, which yields residue terms at $w=1+z$, $w=1-z$, $w=s+\beta-\frac{1}{2}$ and $w=\frac{1}{2}$. The moved integral is $O(x^{-\varepsilon})$.
  
\subsubsection{The residue at $w=1+z$}
  The residue here is:
  \begin{align}
  &\res{w=1+z} \res{s'=w} S_{o_1}^c \notag \\
  = &\lim_{\delta \to 0}  \varphi(Q) G \sum_{l_1,l_2} \conj{\chi(l_1)}\chi(l_2) (l_1l_2)^{\frac{k}{2}+\alpha} l_1^{it}l_2^{-it} \notag \\
  &\times \triinvmellin{\gamma_2}{\gamma_4}{C} \frac{(4\pi)^k 2^{s-\half}}{\Gamma(s+k-1) 2\sqrt{\pi}} \frac{\pi^{\frac{1}{2}-z}}{\Gamma(\frac{1}{2}-z)}\notag \\
  &\times \mathcal{V}_{N[l_1,l_2]} \sum_\mathfrak{a=\frac{b}{c}} \left( \frac{(c,\frac{N}{c})}{N\frac{N}{c}}\right)^{\frac{1}{2}-z}  \frac{\varphi(\frac{N}{c})}{ \varphi((c,\frac{N}{c}))} \frac{\zeta(1+2z)}{\zeta^{(N)}(1-2z)}  \conj{ \langle U,E_\mathfrak{a}(*,\half+z) \rangle}  \notag \\
  &\times  \betaF{\frac{3}{2}+z - s - \beta}{s + \beta+ \tfrac{k}{2} - 1+it}{z+ \tfrac{k+1}{2}+it} M(s,\tfrac{z}{i},\delta) \frac{\Gamma(\beta)\Gamma(G-\beta)}{N^{\frac{1}{2}+z}\Gamma(G)}\notag \\
  &\times v( \tfrac{1}{2}+z) (xl_1)^{\half+z} \prod_{p | c} (1-p^{2z}) Q^{-(1+2z)} \prod_{p^\gamma \| Q} (\sigma_{2z}(p^\gamma) - p^{-1} \sigma_{2z}(p^{\gamma-1})) \,d  z \,d  \beta \,d  s\label{eq_S_o1c_sw_lead}
  \end{align}
  Applying function equation to $\zeta(1+2z)$ and doing some simplifications, we obtain:
  \begin{align}
  &\res{w=1+z} \res{s'=w} S_{o_1}^c \notag \\
  = &\lim_{\delta \to 0}  \varphi(Q) G \sum_{l_1,l_2} \conj{\chi(l_1)}\chi(l_2) (l_1l_2)^{\frac{k}{2}+\alpha} l_1^{it}l_2^{-it} \notag \\
  &\times \triinvmellin{\gamma_2}{\gamma_4}{C} \frac{(4\pi)^k 2^{s-\half}\pi^{\frac{1}{2}+z}Q^{-1}}{\Gamma(s+k-1) 2\sqrt{\pi}}  \!\!\! \prod_{p^\gamma \| Q} \!\! (\sigma_{-2z}(p^\gamma) - p^{-1-2z} \sigma_{-2z}(p^{\gamma-1}))\notag \\
  &\times \mathcal{V}_{N[l_1,l_2]} \sum_\mathfrak{a=\frac{b}{c}}\frac{\sqrt{\pi}\Gamma(-z)}{\Gamma(\frac{1}{2}-z)} \left( \frac{(c,\frac{N}{c})}{N\frac{N}{c}}\right)^{\frac{1}{2}-z}  \frac{\varphi(\frac{N}{c})}{ \varphi((c,\frac{N}{c}))} \frac{\zeta^{(c)}(-2z)}{\zeta^{(N)}(1-2z)}  \conj{ \langle U,E_\mathfrak{a}(*,\half+z) \rangle}  \notag \\
  &\times  \betaF{\frac{3}{2}+z - s - \beta}{s + \beta+ \tfrac{k}{2} - 1+it}{z+ \tfrac{k+1}{2}+it} M(s,\tfrac{z}{i},\delta) \frac{\Gamma(\beta)\Gamma(G-\beta)}{N^{\frac{1}{2}+z}\Gamma(G)\Gamma(\frac{1}{2}+z)}\notag \\
  &\times v( \half+z) (xl_1)^{\half+z}  \,d  z \,d  \beta \,d  s\label{eq_S_o1c_sw_lead2}
  \end{align}
  At this point, we use functional equation of Eisenstein series at $0$-cusp \eqref{FE_Eisen_0}, simplifying this further to:
  \begin{align}
  &\res{w=1+z} \res{s'=w} S_{o_1}^c \notag \\
  = &\lim_{\delta \to 0}  \varphi(Q) G \sum_{l_1,l_2} \conj{\chi(l_1)}\chi(l_2) (l_1l_2)^{\frac{k}{2}+\alpha} l_1^{it}l_2^{-it} \notag \\
  &\times \triinvmellin{\gamma_2}{\gamma_4}{C} \frac{(4\pi)^k 2^{s-\half}\pi^{\frac{1}{2}+z}Q^{-1}}{\Gamma(s+k-1) 2\sqrt{\pi}} \!\!\! \prod_{p^\gamma \| Q} \!\! (\sigma_{-2z}(p^\gamma) - p^{-1-2z} \sigma_{-2z}(p^{\gamma-1}))\notag \\
  &\times \mathcal{V}_{N[l_1,l_2]} \conj{ \langle U,E_0(*,\half-z) \rangle} \betaF{\frac{3}{2}+z - s - \beta}{s + \beta+ \tfrac{k}{2} - 1+it}{z+ \tfrac{k+1}{2}+it} \notag \\
  &\times M(s,\tfrac{z}{i},\delta) \frac{\Gamma(\beta)\Gamma(G-\beta)}{N^{\frac{1}{2}+z}\Gamma(G)\Gamma(\frac{1}{2}+z)} v( \half+z) (xl_1)^{\half+z}  \,d  z \,d  \beta \,d  s\label{eq_S_o1c_sw_sim}
  \end{align}
  We now move $\re z$ to $-\frac{1}{2}-\varepsilon$, picking up poles at $z=s+\beta-\frac{3}{2}$ and $z=\frac{1}{2}-s$. There is technically also a pole at $z=-\frac{1}{2}$, but its residue becomes $0$ when $\delta \to 0$. The moved integral is $O(x^{-\varepsilon})$.
  
  We investigate the residue at $z=s+\beta-\frac{3}{2}$:
  \begin{align}
  &\res{z=s+\beta-\frac{3}{2}} \res{w=1+z} \res{s'=w} S_{o_1}^c \notag \\
  = &\lim_{\delta \to 0}  \varphi(Q) G \sum_{l_1,l_2} \conj{\chi(l_1)}\chi(l_2) (l_1l_2)^{\frac{k}{2}+\alpha} l_1^{it}l_2^{-it} \dblinvmellin{\gamma_2}{\gamma_4} \frac{(4\pi)^k 2^{s-\half}\pi^{s+\beta-1}}{\Gamma(s+k-1) 2\sqrt{\pi}} \notag \\
  &\times Q^{-1} \prod_{p^\gamma \| Q} (\sigma_{3-2s-2\beta}(p^\gamma) - p^{2-2s-2\beta} \sigma_{3-2s-2\beta}(p^{\gamma-1}))M(s,\tfrac{s+\beta-\frac{3}{2}}{i},\delta)\notag \\
  &\times \mathcal{V}_{N[l_1,l_2]} \conj{ \langle U,E_0(*,2-s-\beta) \rangle}   \frac{\Gamma(\beta)\Gamma(G-\beta)v(s+\beta-1)}{N^{s+\beta-1}\Gamma(G)\Gamma(s+\beta-1)}  (xl_1)^{s+\beta-1}  \,d  \beta \,d  s\label{eq_S_o1c_key}
  \end{align}
  We move $\re \beta$ to $\frac{1}{4}-\varepsilon$, hitting a pole from the $M$-function at $\beta = 2-2s$. Again, there is technically a pole at $\beta=1-s$, but the residue vanishes upon taking $\delta \to 0$. The moved integral is again $O(x^{-\varepsilon})$. Taking $\delta \to 0$, the residue at $\beta = 2-2s$ is:
  \begin{align}
  &\res{\beta=2-2s} \res{z=s+\beta-\frac{3}{2}} \res{w=1+z} \res{s'=w} S_{o_1}^c \notag \\
  = &\varphi(Q) G \sum_{l_1,l_2} \conj{\chi(l_1)}\chi(l_2) (l_1l_2)^{\frac{k}{2}+\alpha} l_1^{it}l_2^{-it} \invmellin{\gamma_2} \frac{(4\pi)^k \pi^{1-s}}{2\Gamma(s+k-1)} \notag \\
  &\times Q^{-1} \prod_{p^\gamma \| Q} (\sigma_{2s-1}(p^\gamma) - p^{2s-2} \sigma_{2s-1}(p^{\gamma-1})) \frac{\Gamma(2s-1)}{\Gamma(s)}\notag \\
  &\times \mathcal{V}_{N[l_1,l_2]} \conj{ \langle U,E_0(*,s) \rangle}   \frac{\Gamma(2-2s)\Gamma(G-2+2s)v(1-s)}{N^{1-s}\Gamma(G)\Gamma(1-s)}  (xl_1)^{1-s}  \,d  s\label{eq_S_o1c_key2}
  \end{align}
  We move $\re s$ up to $1+\varepsilon$, encountering a double pole at $s=1$. The moved integral is $O(x^{-\varepsilon})$. The residue here is:
  \begin{align}
    &\varphi(Q) G \sum_{l_1,l_2} \conj{\chi(l_1)}\chi(l_2) (l_1l_2)^{\frac{1}{2}+\alpha} l_1^{it}l_2^{-it} \lt( - \frac{(4\pi)^k}{4\Gamma(k)} \langle f, f\rangle b_{l_1,l_2} \log(xl_1) + c_1 \rt),
  \end{align}
  where $b_{l_1,l_2}$ is as defined from \eqref{eq_defn_b_l1l2} and
  \begin{align*}
    c_1 = \res{s=1} \Big(& \frac{ (4\pi)^k \pi^{1-s}}{2\Gamma(s+k-1)} Q^{-1} \prod_{p^\gamma \| Q} (\sigma_{2s-1}(p^\gamma) - p^{2s-2} \sigma_{2s-1}(p^{\gamma-1})) \frac{\Gamma(2s-1)}{\Gamma(s)} \\
  &\qqquad \times \mathcal{V}_{N[l_1,l_2]} \conj{ \langle U,E_0(*,s) \rangle}   \frac{\Gamma(2-2s)\Gamma(G-2+2s)v(1-s)}{N^{1-s}\Gamma(G)\Gamma(1-s)}  \Big)
  \end{align*}
  The one thing we need to know about $c_1$ is that it is $O(Q^{\varepsilon} G^{\varepsilon} \mathcal{L}^{1-k+\varepsilon})$.
  
  Continuing the investigating of residue terms, we look at the residue at $z=\frac{1}{2}-s$:
  \begin{align}
  &\res{z=\frac{1}{2}-s} \res{w=1+z} \res{s'=w} S_{o_1}^c \notag \\
  = &\varphi(Q) G \sum_{l_1,l_2} \conj{\chi(l_1)}\chi(l_2) (l_1l_2)^{\frac{k}{2}+\alpha} l_1^{it}l_2^{-it} \notag \\
  &\times \dblinvmellin{\gamma_2}{\gamma_4} \frac{(4\pi)^k \pi^{1-s}}{2\Gamma(s+k-1)} Q^{-1} \prod_{p^\gamma \| Q} (\sigma_{2s-1}(p^\gamma) - p^{2s-2} \sigma_{2s-1}(p^{\gamma-1}))\notag \\
  &\times \mathcal{V}_{N[l_1,l_2]} \conj{ \langle U,E_0(*,s) \rangle} \betaF{2 - 2s - \beta}{s + \beta+ \tfrac{k}{2} - 1+it}{ \tfrac{k}{2}+1-s+it} \notag \\
  &\times \frac{\Gamma(2s-1)}{\Gamma(s)}\frac{\Gamma(\beta)\Gamma(G-\beta)}{N^{1-s}\Gamma(G)\Gamma(1-s)} v(1-s) (xl_1)^{1-s}\,d  \beta \,d  s\label{eq_S_o1c_sw_2nd}
  \end{align}
  We move $\re s$ up to $1+\varepsilon$, hitting a simple pole at $s=1$. The moved integral is $O(x^{-\varepsilon})$. The residue at $s=1$ is:
  \begin{align}
  &\res{s=1} \res{z=\frac{1}{2}-s} \res{w=1+z} \res{s'=w} S_{o_1}^c \notag \\
  = &\varphi(Q) G \sum_{l_1,l_2} \conj{\chi(l_1)}\chi(l_2) (l_1l_2)^{\frac{1}{2}+\alpha} l_1^{it}l_2^{-it} \notag \\
  &\times \invmellin{\gamma_4} \frac{(4\pi)^k}{2\Gamma(k)} \langle f, f\rangle b_{l_1,l_2} \betaF{- \beta}{\beta+ \tfrac{k}{2} +it}{ \tfrac{k}{2}+it} \frac{\Gamma(\beta)\Gamma(G-\beta)}{\Gamma(G)} \,d  \beta \notag \\
  \ll& Q \mathcal{L}^{1+2\alpha+\varepsilon} G^{1+\varepsilon} \label{eq_S_o1c_sw_2nd_d}
  \end{align}
  Putting together all the cases here, we have:
  \begin{lemma}
  For $G \asymp (1+|t|)^\frac{2}{3-2\theta} \log^5 Q$,
\begin{align}
 & \res{w=1+z}\res{s' = w} S_{o_1}^c  \notag \\
 = &-\varphi(Q) G \sum_{l_1,l_2} \conj{\chi(l_1)}\chi(l_2) (l_1l_2)^{\frac{1}{2}+\alpha} l_1^{it}l_2^{-it} (4\pi)^k \langle f,f \rangle \frac{\log x}{4\Gamma(k)} \cdot b_{l_1,l_2} \notag \\
  &+ O(G^{1+\varepsilon} Q^{1+\varepsilon} L^{1+2\alpha+\varepsilon}) + O(x^{-\varepsilon}), \label{prop_w_+z}
\end{align}
where $b_{l_1,l_2}$ is defined as in \eqref{eq_defn_b_l1l2}.
\end{lemma}

\subsubsection{Residue at $w=1-z$}
Looking at the $\zeta_{\mathfrak{a},Q}(w,-z)$, it turns out that the pole only occurs at the $0$-cusp. Hence, the residue term here is:
  \begin{align}
  &\res{w=1-z} \res{s'=w} S_{o_1}^c \notag \\
  = &\lim_{\delta \to 0}  \varphi(Q) G \sum_{l_1,l_2} \conj{\chi(l_1)}\chi(l_2) (l_1l_2)^{\frac{k}{2}+\alpha} l_1^{it}l_2^{-it} \notag \\
  &\times \triinvmellin{\gamma_2}{\gamma_4}{C} \frac{(4\pi)^k 2^{s-\half}Q^{-1}}{\Gamma(s+k-1) 2\sqrt{\pi}} \frac{\pi^{\frac{1}{2}-z}}{\Gamma(\frac{1}{2}-z)} \left( \frac{1}{N}\right)^{\frac{1}{2}-z} \frac{\Gamma(\beta)\Gamma(G-\beta)}{\Gamma(G)}\notag \\
  &\times \mathcal{V}_{N[l_1,l_2]} M(s,\tfrac{z}{i},\delta)  \conj{ \langle U,E_0(*,\half+z) \rangle} \prod_{p^\gamma \| Q} (\sigma_{2z}(p^\gamma) - p^{-(1-2z)} \sigma_{2z}(p^{\gamma-1})) \notag \\
  &\times  \betaF{\frac{3}{2}-z - s - \beta}{s + \beta+ \tfrac{k}{2} - 1+it}{-z+ \tfrac{k+1}{2}+it}  v( \tfrac{1}{2}-z) (xl_1)^{\half-z}  \,d  z \,d  \beta \,d  s \label{eq_S_o1c_sp_lead}
  \end{align}
  Moving the $z$-line of integration to $-C$, and changing variable from $z \mapsto -z$, we see that this term is exactly the same as the residue at $w=1+z$.

\subsubsection{Residue at $w=s+\beta-\frac{1}{2}$}
\begin{align}
  &\res{w=s+\beta-\frac{1}{2}} \res{s'=w} S_{o_1}^c \notag \\
  = &\lim_{\delta \to 0}  \varphi(Q) G \sum_{l_1,l_2} \conj{\chi(l_1)}\chi(l_2) (l_1l_2)^{\frac{k}{2}+\alpha} l_1^{it}l_2^{-it} \notag \\
  &\times \triinvmellin{\gamma_2}{\gamma_4}{C} \frac{(4\pi)^k 2^{s-\half}}{\Gamma(s+k-1) 2\sqrt{\pi}} \frac{\Gamma(\beta)\Gamma(G-\beta)}{\Gamma(G)}v( s+\beta-1)\notag \\
  &\times \mathcal{V}_{N[l_1,l_2]} \sum_\mathfrak{a} \zeta_{\mathfrak{a},Q}(s+\beta-\frac{1}{2},-z) M(s,\tfrac{z}{i},\delta)  \conj{ \langle U,E_\mathfrak{a}(*,\half+z) \rangle}  (xl_1)^{s+\beta-1}  \,d  z \,d  \beta \,d  s \label{eq_S_o1c_sp_be}
\end{align}
Here we move $\re \beta$ to $\frac{1}{4}-\varepsilon$, picking up a simple pole at $\beta = 1-s$. The moved integral is again $O(x^{-\varepsilon})$. The residue is:
\begin{align}
  &\res{\beta=1-s} \res{w=s+\beta-\frac{1}{2}} \res{s'=w} S_{o_1}^c \notag \\
  = &\lim_{\delta \to 0}  \varphi(Q) G \sum_{l_1,l_2} \conj{\chi(l_1)}\chi(l_2) (l_1l_2)^{\frac{k}{2}+\alpha} l_1^{it}l_2^{-it} \notag \\
  &\times \dblinvmellin{\gamma_2}{C} \frac{(4\pi)^k 2^{s-\half}}{\Gamma(s+k-1) 2\sqrt{\pi}} \frac{\Gamma(1-s)\Gamma(G-1+s)}{\Gamma(G)} \notag \\
  &\times \mathcal{V}_{N[l_1,l_2]} \sum_\mathfrak{a} \zeta_{\mathfrak{a},Q}(\frac{1}{2},-z) M(s,\tfrac{z}{i},\delta)  \conj{ \langle U,E_\mathfrak{a}(*,\half+z) \rangle} \,d  z \,d  s \label{eq_S_o1c_sp_be_p}
\end{align}
In order to take the $\delta$-limit, we move $\re s$ down to $\frac{1}{2}-\frac{k}{2}-\varepsilon$, encountering poles at $s=\frac{1}{2} \pm z - r$, where $r$ is an integer such that $0 \leq r \leq \frac{k}{2}$. For $r$ being such an integer, taking limit as $\delta \to 0$, we have the residue:
\begin{align}
  &\lt( \res{s=\frac{1}{2}+z-r} + \res{s=\frac{1}{2}-z-r} \rt) \res{\beta=1-s} \res{w=s+\beta-\frac{1}{2}} \res{s'=w} S_{o_1}^c \notag \\
  = & \varphi(Q) G \sum_{l_1,l_2} \conj{\chi(l_1)}\chi(l_2) (l_1l_2)^{\frac{k}{2}+\alpha} l_1^{it}l_2^{-it} \notag \\
  &\times \invmellin{C} \frac{(-1)^r (4\pi)^k}{2r!} \mathcal{V}_{N[l_1,l_2]} \sum_\mathfrak{a} \zeta_{\mathfrak{a},Q}(\frac{1}{2},-z) \conj{ \langle U,E_\mathfrak{a}(*,\half+z) \rangle} \notag \\
  &\times  \Big( \frac{\Gamma(2z-r) \Gamma(\frac{1}{2}-z+r) \Gamma(\frac{1}{2}-z+r)\Gamma(G-\frac{1}{2}+z-r)}{\Gamma(\frac{1}{2}+z)\Gamma(\frac{1}{2}-z) \Gamma(k-\frac{1}{2}+z-r)\Gamma(G)} \notag \\
  &\qquad + \frac{\Gamma(-2z-r) \Gamma(\frac{1}{2}+z+r) \Gamma(\frac{1}{2}+z+r)\Gamma(G-\frac{1}{2}-z-r)}{\Gamma(\frac{1}{2}+z)\Gamma(\frac{1}{2}-z) \Gamma(k-\frac{1}{2}-z-r)\Gamma(G)} \Big)\,d  z \label{eq_S_o1c_sp_beps}
\end{align}
We see that there is enough exponential decay to guarantee convergence, and that this term is $O(G^{\frac{1}{2}-r+\varepsilon} Q^{\frac{1}{2}+\varepsilon} \mathcal{L}^{3+2\alpha+\varepsilon})$.

For the moved integral, it is treated in the same way as \eqref{eq_S_o1d_lead_moved}, and can be shown to be $O(Q^{\frac{1}{2}+\varepsilon} G^{\frac{1}{2}-\frac{k}{2}-\varepsilon} \mathcal{L}^{3+2\alpha +\varepsilon})$

\subsubsection{The residue at $w=\frac{1}{2}$}
The residue is:
 \begin{align}
  &\res{w=\frac{1}{2}} \res{s'=w} S_{o_1}^c \notag \\
  = &\lim_{\delta \to 0}  \varphi(Q) G \sum_{l_1,l_2} \conj{\chi(l_1)}\chi(l_2) (l_1l_2)^{\frac{k}{2}+\alpha} l_1^{it}l_2^{-it} \notag \\
  &\times \triinvmellin{\gamma_2}{\gamma_4}{C} \frac{(4\pi)^k 2^{s-\half}\mathcal{V}_{N[l_1,l_2]}}{\Gamma(s+k-1) 2\sqrt{\pi}} \sum_\mathfrak{a} \zeta_{\mathfrak{a},Q}(\frac{1}{2},-z)   \conj{ \langle U,E_\mathfrak{a}(*,\half+z) \rangle}\notag \\
  &\times M(s,\tfrac{z}{i},\delta) \betaF{1 - s - \beta}{s + \beta+ \tfrac{k}{2} - 1+it}{\tfrac{k}{2}+it}  \frac{\Gamma(\beta)\Gamma(G-\beta)}{\Gamma(G)} \,d  z \,d  \beta \,d  s  \label{eq_S_o1c_sp_Lw}
  \end{align}
Now we move $\re s$ down to $\frac{1}{2}-\frac{k}{2}-\theta-\varepsilon$ to enable us to take the $\delta$-limit. In doing so, we pick up poles at $s=1-\beta$, $s=1-\frac{k}{2}-\beta-it$, and $s=\frac{1}{2}\pm z -r$, where $0 \leq r \leq \frac{k}{2}$ is an integer.

The residue at $s=1-\beta$ is:
 \begin{align}
  &\res{s=1-\beta} \res{w=\frac{1}{2}} \res{s'=w} S_{o_1}^c \notag \\
  = -&\lim_{\delta \to 0}  \varphi(Q) G \sum_{l_1,l_2} \conj{\chi(l_1)}\chi(l_2) (l_1l_2)^{\frac{k}{2}+\alpha} l_1^{it}l_2^{-it} \notag \\
  &\times \dblinvmellin{\gamma_4}{C} \frac{(4\pi)^k 2^{\half-\beta}}{\Gamma(k-\beta) 2\sqrt{\pi}} M(1-\beta,\tfrac{z}{i},\delta) \notag \\
  &\times \mathcal{V}_{N[l_1,l_2]} \sum_\mathfrak{a} \zeta_{\mathfrak{a},Q}(\frac{1}{2},-z)  \conj{ \langle U,E_\mathfrak{a}(*,\half+z) \rangle}     \frac{\Gamma(\beta)\Gamma(G-\beta)}{\Gamma(G)} \,d  z \,d  \beta \label{eq_S_o1c_sp_Lw_Hs}
  \end{align}
  Changing variable $\beta =1-s$, we see that the contour line becomes $\re s = \frac{1}{2}-\theta-\varepsilon$. This is pretty much the same term as \eqref{eq_S_o1c_sp_be_p} and the bound there applies.
  
Taking the $\delta$-limit, the analysis of the residue at $s=1-\frac{k}{2}-\beta-it$ is very similar to that of \eqref{eq_S_o1d_2ndwLs}, yielding the bound $O(G^{\half+\varepsilon} Q^{\half + \varepsilon} \mathcal{L}^{3 + 2\alpha + \varepsilon})$.
 
%
%
%

The residues at $s=\frac{1}{2} \pm z -r $ are:
 \begin{align}
  &\lt( \res{s=\frac{1}{2}+z-r} + \res{s=\frac{1}{2}-z-r} \rt) \res{w=\frac{1}{2}} \res{s'=w} S_{o_1}^c \notag \\
  = &\lim_{\delta \to 0} \frac{(4\pi)^k }{2\sqrt{\pi}}  \varphi(Q) G \sum_{l_1,l_2} \conj{\chi(l_1)}\chi(l_2) (l_1l_2)^{\frac{k}{2}+\alpha} l_1^{it}l_2^{-it} \notag \\
  &\times \dblinvmellin{\gamma_4}{C} \mathcal{V}_{N[l_1,l_2]}\sum_\mathfrak{a} \zeta_{\mathfrak{a},Q}(\frac{1}{2},-z)\conj{ \langle U,E_\mathfrak{a}(*,\half+z) \rangle} \frac{\Gamma(\beta)\Gamma(G-\beta)}{\Gamma(G)}\notag \\
  &\times  \Big( \frac{c_r(z,\delta) 2^{z-r}}{\Gamma(k-\frac{1}{2}+z-r)}\betaF{\frac{1}{2}-z+r - \beta}{z-r + \beta+ \tfrac{k-1}{2}+it}{\tfrac{k}{2}+it} \notag \\
  &\qquad + \frac{c_r(-z,\delta) 2^{-z-r}}{\Gamma(k-\frac{1}{2}-z-r)}\betaF{\frac{1}{2}+z+r - \beta}{-z-r + \beta+ \tfrac{k-1}{2}+it}{\tfrac{k}{2}+it} \Big) \,d  z \,d  \beta \label{eq_S_o1c_sp_Lws}
  \end{align}
  Similar to the previous section, we can prove the following proposition as in lemma \ref{propo_crazy} with slight modifications:
\begin{lemma} \label{propo_crazy3}
For $G \asymp (1+|t|)^{\frac{2}{3-2\theta}} \log^5 Q$ and $\re \beta = a = \frac{3}{2} + 2r + \varepsilon$,
\begin{align}
  &\lim_{\delta \to 0} G \dblinvmellin{a}{C} \mathcal{V}_{N[l_1,l_2]} \sum_\mathfrak{a} \zeta_{\mathfrak{a},Q}(\half,-z) \conj{ \langle U,E_\mathfrak{a}(*,\half+z) \rangle} \frac{\Gamma(\beta)\Gamma(G-\beta)}{\Gamma(G)}\notag \\
  &\times  \Big( \frac{2^{z-r}c_r(z,\delta)}{2\sqrt{\pi}\Gamma(k-\half + z -r)}  \betaF{\half - z + r - \beta}{ \beta + \tfrac{k-1}{2} + z - r+it}{\tfrac{k}{2}+it} \notag \\
  &\phantom{\times} + \frac{2^{-z-r}c_r(-z,\delta) }{2\sqrt{\pi}\Gamma(k-\half - z -r)} \betaF{\half + z + r -\beta}{ \beta + \tfrac{k-1}{2} - z - r+it}{\tfrac{k}{2}+it} \Big) \,d  z \,d  \beta \notag \\
  \ll & G^{1+\varepsilon} \mathcal{L}^{1-k+\varepsilon} Q^{-\half + \varepsilon} \times [l_1,l_2]^{-\frac{1}{2}} \label{eq_cts_res_bound}
\end{align}
\end{lemma}
By the same arguments that we made after proving lemma \ref{propo_crazy}, we can conclude that the double integral \eqref{eq_S_o1c_sp_Lws} is $O(G^{1+\varepsilon} Q^{\half + \varepsilon} \mathcal{L}^{2 + 2\alpha + \varepsilon})$.

For the moved integral, we estimate with the following lemma, which can be proved in the same way as lemma \ref{propo_crazy}:
\begin{lemma} \label{propo_crazy4}
For $G \asymp (1+|t|)^{\frac{2}{3-2\theta}} \log^5 Q$ and $\re \beta = a = \frac{5}{2} + k + 2\varepsilon$,
\begin{align}
  &\lim_{\delta \to 0} G \triinvmellin{\gamma'_2}{\gamma_4}{C} \frac{(4\pi)^k \mathcal{V}_{N[l_1,l_2]} 2^{s-\half}}{\Gamma(s+k-1) 2\sqrt{\pi}}  \notag \\
  &\qquad \times  \sum_\mathfrak{a} \zeta_{\mathfrak{a},Q}(\half,-z) M(s,\tfrac{z}{i},\delta) \conj{ \langle U,E_\mathfrak{a}(*,\half+z) \rangle} \notag \\
  &\qquad \times \betaF{1 - s-\beta}{s +\beta+\tfrac{k}{2} - 1+it}{\tfrac{k}{2}+it} \frac{\Gamma(\beta)\Gamma(G-\beta)}{\Gamma(G)} \,d  z \,d  \beta \,d  s \notag \\
  \ll & G^{1+\varepsilon} \mathcal{L}^{1-k+\varepsilon} Q^{-\half+\varepsilon} [l_1,l_2]^{-\frac{1}{2}}
\end{align}
\end{lemma}
By the same arguments that we made after stating lemma \ref{propo_crazy2}, we can conclude that the moved triple integral \eqref{eq_S_o1c_sp_Lw} is $O(G^{1+\varepsilon} Q^{\half+\varepsilon} L^{2+2\alpha+\varepsilon})$.

All these results combined together gives proposition \ref{propo_total_cts}.
\end{proof}

\subsection{The residue at $s'=1 \pm z$}
Calculating the contribution of the residue at $s'=1+z$ is very similar to the subcase $w=1+z$ in the previous section, except that we do not have any poles coming from the $v$ functions. As such, we will omit the details and only state the result here:
\begin{lemma}
  For $G \asymp (1+|t|)^\frac{2}{3-2\theta} \log^5 Q$,
\begin{align}
 & \res{s' = 1+z} S_{o_1}^c  \ll G^{1+\varepsilon} Q^{1+\varepsilon} \mathcal{L}^{1+2\alpha+\varepsilon} + x^{-\varepsilon}, \label{prop_s_+z}
\end{align}
\end{lemma}
The residue at $s'=1-z$ can be shown to be the same as $s'=1+z$, and hence has the same contribution as above.
\subsection{The residue at $s'=s+\beta-\frac{1}{2}$}
The residue is:
\begin{align}
  &\res{s'=s+\beta-\frac{1}{2}} S_{o_1}^c \notag \\
  = &\lim_{\delta \to 0}  \varphi(Q) G \sum_{l_1,l_2} \conj{\chi(l_1)}\chi(l_2) (l_1l_2)^{\frac{k}{2}+\alpha} l_1^{it}l_2^{-it} \notag \\
  &\times \quadinvmellin{\gamma_1}{\gamma_2}{\gamma_4}{C} \frac{(4\pi)^k 2^{s-\half}}{\Gamma(s+k-1) 2\sqrt{\pi}} \notag \\
  &\times \mathcal{V}_{N[l_1,l_2]} \sum_\mathfrak{a}  \zeta_{\mathfrak{a},Q}(s+\beta-\frac{1}{2},-z) M(s,\tfrac{z}{i},\delta)  \conj{ \langle U,E_\mathfrak{a}(*,\half+z) \rangle}  \notag \\
  &\times v( s+\beta-\frac{1}{2}-w) v( w-\half) x^{s+\beta-1} l_2^{s+\beta-\frac{1}{2}-w} l_1^{w-\half} \frac{\Gamma(\beta)\Gamma(G-\beta)}{\Gamma(G)} \,d  z \,d  \beta  \,d  s \,d  w \label{eq_S_o1c_Lsp}
\end{align}
Moving the line of integration of $\re \beta$ to $\frac{1}{4}-\varepsilon$ does not pass through any poles, and hence this is $O(x^{-\varepsilon})$.

\section{Proof of Theorem \ref{thm_main}}
Putting the results of the sections together (in particular, propositions \ref{eq_S_d1_prop}, \ref{eq_S_d2_propo}, \ref{eq_S_o1d_Zres_propo}, \ref{eq_S_o1d_Z_sp2}, and \ref{propo_total_cts}), we obtain that $S$ as defined in $\eqref{eq_S_defn}$ has the following bound:
\[ S = O(G^{1+\varepsilon} Q^{1+\varepsilon} \mathcal{L}^{1+2\alpha + \varepsilon}) + O(G^{1+\varepsilon} Q^{\half + \theta+\varepsilon} \mathcal{L}^{3+2\alpha+ \varepsilon}) + O(x^{-\varepsilon}), \]
where $G \asymp (1+|t|)^{\frac{2}{3-2\theta}} \log^5 Q$ and $\alpha = \frac{1}{\log (Q(1+|t|))}$.

Plugging this into the right-hand side of proposition \ref{propo_amp}, we have:
\begin{align*}
 &|L(\half + it, f_\chi)|^2 |\sum_{l \sim \mathcal{L}} 1|^2 \\
 \ll & (1+|t|)^{\frac{2}{3-2\theta}+\varepsilon} \lt( Q^{1+\varepsilon} \mathcal{L}^{1+2\alpha + \varepsilon} + Q^{\half +\theta+ \varepsilon} \mathcal{L}^{3+2\alpha+ \varepsilon} \rt) + x^{-\varepsilon} + Q\mathcal{L} + (1+|t|)^\varepsilon Q^{1+\varepsilon}
\end{align*}

Taking $x \to \infty$, we can drop the $x$-term above. Note that
\[ \sum_{\substack{l \sim \mathcal{L} \\ l \text{ prime} \\ (l,QN) = 1}} 1 \asymp \frac{\mathcal{L}}{\log \mathcal{L}}. \]

Hence, we can conclude that:
\begin{align*}
 |L(\half + it, f_\chi)|^2 \ll &(1+|t|)^{\frac{2}{3-2\theta}+\varepsilon} \lt( Q^{1+\varepsilon} \mathcal{L}^{-1+2\alpha + \varepsilon} +  Q^{\half +\theta+ \varepsilon} \mathcal{L}^{1+2\alpha+ \varepsilon} \rt)
\end{align*}

In order to balance the effects of the two terms, we set $L = Q^{\frac{1}{4} - \frac{\theta}{2} + \varepsilon}$. Recalling $\alpha = \frac{1}{\log (Q(1+|t|))}$, we have:
\[ |L(\half + it, f_\chi)|^2 \ll (1+|t|)^{\frac{2}{3-2\theta}+\varepsilon} Q^{\frac{3}{4}+\frac{\theta}{2}+\varepsilon} \]

Taking square roots of the above, this is the theorem. 

%

\bibliographystyle{abbrv}
\bibliography{bibfile}

\end{document}